\RequirePackage{fix-cm}
\documentclass[smallextended]{svjour3}       % onecolumn (second format)
\smartqed  % flush right qed marks, e.g. at end of proof
\usepackage{graphicx}
\usepackage[round,authoryear]{natbib}
\usepackage{amssymb}
\usepackage{amsmath}
\usepackage{amsbsy}
\usepackage{theorem}
\usepackage{graphicx}
\usepackage{epstopdf}
\usepackage{caption}
\journalname{TEST}
\begin{document}
\bibliographystyle{plainnat}
\title{On a class of minimum contrast estimators for Gegenbauer random fields}


\author{Rosa M. Espejo\and
        Nikolai N. Leonenko \and Andriy  Olenko \and Mar\'{\i}a D. Ruiz-Medina
}


\institute{R.M. Espejo \at
Department of Statistics and Operation Research,  Faculty of Sciences, University of Granada,
Spain\\
              \email{rosaespejo@ugr.es}        \and
N.N. Leonenko \at
Cardiff School of Mathematics, University of Cardiff, United Kingdom\\
\email{LeonenkoN@cardiff.ac.uk} \and 
A. Olenko \at
Department of Mathematics and Statistics, La Trobe University, Australia\\ 
Tel.: +613-94792609\\
              Fax: +613-94792466\\
\email{a.olenko@latrobe.edu.au} \and
M.D. Ruiz-Medina \at
Department of Statistics and Operation Research,  Faculty of Sciences, University of Granada,
Spain\\
\email{mruiz@ugr.es} }

\date{Received: date / Accepted: date}

\maketitle

\begin{abstract}
The article introduces spatial long-range dependent models based on the fractional difference operators associated with the Gegenbauer polynomials.
The results on consistency and asymptotic normality of a class of minimum
contrast estimators of long-range dependence parameters of the models are obtained. A methodology to verify assumptions for consistency and asymptotic normality of minimum contrast estimators is developed. Numerical results are presented to confirm the theoretical findings. 
\end{abstract}
\clearpage
\setcounter{page}{1}

\section{Introduction}
Among the extensive literature on long-range dependence, relatively few publications are devoted to cyclical long-memory processes or long-range dependent random fields. However, models with singularities at non-zero frequencies are of great importance in applications. For example, many time series show cyclical/seasonal evolutions. Singularities at non-zero frequencies produces peaks in the spectral density whose locations define periods of the cycles.  A survey of some recent asymptotic results for cyclical  long-range dependent random processes and  fields can be found in \citet{Iva:2013}  and \citet{Ole:2013}. 

In image analysis popular isotropic spatial processes with singularities of the spectral density at non-zero frequencies are wave, $J$-Bessel, and Gegenbauer models. \citet{Espejo:2014} investigated  probabilistic properties of spatial Gegenbauer models.
A realization of the Gegenbauer random field on $100\times100$ grid is shown in Figure~\ref{fig1}.\vspace{-0.7cm}
\begin{figure}[htbp]
  \begin{center}
  \includegraphics[width=0.6\textwidth,trim=20 10 0 10, clip=true]{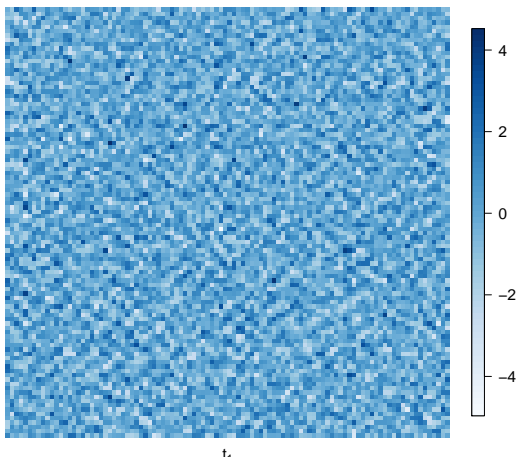}  
\caption{Simulated realization of the Gegenbauer random  field.}\vspace{-0.7cm} \label{fig1}
  \end{center}
\end{figure} 

This article studies minimum contrast estimators (MCEs) of parameters of the Gegenbauer random  fields. The MCE methodology has been widely applied in different statistical
frameworks (see, for example, \citet{Anh:2004,Anh:2007};
\citet{WeiLin:2012}). One of the first works which used a
MCE methodology for the parameter estimation of spectral
densities of stationary processes was the paper by \citet{Taniguchi:1987}.
\citet{Guyon:1995} introduced  a class of MCEs for
random fields. \citet{Anh:2004} derived consistency
and asymptotic normality of a class of MCEs
for stationary processes within the class of fractional Riesz-Bessel
motion (see \citet{Anh:1999}). Results based
on the second and third-order cumulant spectra were given by \citet{Anh:2007}. They also provided asymptotic properties
of second and third-order sample spectral functionals. These properties are of independent interest, since they
can be applied to study the limiting properties of
nonparametric estimators of processes with short or long-range
dependence.  \citet{WeiLin:2012} applied the minimum
contrast parameter estimation to approximate the drift parameter
of the Ornstein-Uhlenbeck process, when the corresponding stochastic
differential equation is driven by the fractional Brownian motion with a
specific Hurst index.

A burgeoning literature on spatio-temporal estimation has emerged in recent decades 
(see \citet{Beran:2009}, \citet{Chan:2012}, \citet{Giraitis:2001}, \citet{Guo:2009}, \citet{Li:1986}, \citet{Reisen:2006}, among
others). One of the most popular estimation tools applied was the
maximum likelihood estimation method~(MLE).  \citet{Reisen:2006}
addressed the problem of parameter estimation   of fractionally
integrated processes with seasonal components. In order to estimate
the fractional parameters, they propose several log-periodogram
regression estimators with different bandwidths selected
around and/or between the seasonal frequencies. The same methodology
was used by \citet{Li:1986} for fractionally differenced
autoregressive-moving average processes in the stationary time
series context. Several contributions have also been made for MLE of long
memory spatial processes  (see, for example,  \citet{Anh:1995}).  For
two-dimensional spatial data the paper by \citet{Basu:1993}   introduced a spatial unilateral first-order autoregressive
moving average (ARMA)  model. To  implement
 MLE they provided a proper treatment to
border cell values with a substantial effect in estimation of
parameters. \citet{Beran:2009} addressed the problem of the
least-squares estimation of  autoregressive fractionally integrated
moving-average (FARIMA) processes with long-memory. \citet{Cohen:2002} investigated asymptotic properties of least-squares
estimators in  regression models for two-dimensional random fields.
Maximization of the Whittle likelihood has been also considered in the
recent  literature on the MCE (see
for example, \citet{Chan:2012}, \citet{Boissy:2005}, \citet{Leonenko:2006}).  \citet{Leonenko:2006} gave a continuous version of the Whittle contrast
functional supplied with a specific weight function for the
estimation of continuous-parameter stochastic processes, deriving
the  consistency and asymptotic normality of such estimators. \citet{Guo:2009} demonstrated that the Whittle maximum likelihood
estimator is consistent and asymptotically normal for stationary
seasonal autoregressive fractionally integrated moving-average
(SARFIMA) processes.

Parameter estimation of stationary  Gegenbauer
random processes was considered by numerous authors, see, for example, \citet{Gray:1989}, \citet{Chung:1996a,Chung:1996b}, \citet{Woodward:1998}, \citet{Collet:2006}, \citet{McElroy:2012}. \citet{Gray:1989} used  the generating function
of the Gegenbauer polynomials to develop long
memory Gegenbauer autoregressive moving-average (GARMA) models that
generalize the FARIMA process.  GARMA models were
estimated by applying the MLE methodology. \citet{Chung:1996a}
also applied this methodology with slight modifications based on
the conditional sum of squares method. \citet{Chung:1996b} extended
these results to the two-parameter context within the GARMA process
class. \citet{Woodward:1998} introduced a $k$-factor extension of the
GARMA model that allowed to  associate the long-memory behavior 
with each one of the $k$ Gegenbauer frequencies involved.

In this article we restrict our consideration to the estimation of long-range dependence parameters. It is motivated in part by cyclic processes, for which pole locations are known. Also in some applications the spectral density singularity location can be estimated in advance. Various methods, including semiparametric, wavelet, and pseudo-maximum likelihood techniques, of the estimation of a singularity location were discussed by, for example,  \citet{Arteche:2000}, \citet{Giraitis:2001}, and \citet{Ferrara:2001}.

This paper introduces and studies the MCE of parameters of spatial
Gegenbauer  processes. Specifically,  analogous of continuous-space results by \citet{Anh:2004} are formulated for random fields defined on integer grids. The consistency and asymptotic normality of the MCE are obtained using a spatial discrete version of the Ibragimov contrast function. The article develops a methodology to practically verify general theoretical assumptions for consistency and asymptotic normality of MCEs for specific models. The results provide a rigorous platform to conduct model selection and statistical inference.

The outline of the article is the following. In section 2, we start by introducing the main notations of the paper. Some fundamental definitions and the main
results of this article, Theorem 1 and 2, are given in \S 3. Section 4 consists of the proofs of
the  main results. Section 5 presents simulation studies which support the theoretical findings. The Appendix provides  auxiliary
materials that specify for our case the conditions that ensure the consistency and
asymptotic normality of the MCE based on the Ibragimov
contrast function
formulated in \citet{Anh:2004}. 

In what follows we use the symbol $C$ to denote constants which are not important for our discussion. Moreover, the same symbol $C$ may be used for different constants appearing in the same proof.

 All calculations in the article were performed using  the software {\it R} version~3.0.2 and {\it Maple 16,} Maplesoft.
 
\section{Gegenbauer random fields}
This section introduces some of the main definitions of the Gegenbauer random fields given in \citet{Espejo:2014}
(see also, \citet{Chung:1996a,Chung:1996b}, \citet{Gray:1989},  and \citet{Woodward:1998}, for the temporal case).

 Let $Y_{t_{1},t_{2}},$  $(t_{1},t_{2})\in
\mathbb{Z}^{2},$ be a  random field defined on the   grid lattice $\mathbb{Z}^{2}.$ Consider the fractional difference operator $\nabla^{d}_{u}$
defined by
\begin{align}\label{fog}
\nabla^{d}_{u}=(I-2uB+B^{2})^{d}=(1-2\cos\nu B+B^{2})^{d}=[(1-e^{i\nu}B)(1-e^{-i\nu}B)]^{d},
\end{align}
where $B$ is the backward-shift operator, $u=\cos \nu,$ i.e. $\nu=\arccos(u),$ $|u| \leq 1,$ and $d \in
\left(-\frac{1}{2}, \frac{1}{2}\right).$ Assume that $Y$ satisfies the following state
equation  
 \begin{equation}\label{eqGegenbauer}
\nabla^{d_{1}}_{u_{1}}\circ \nabla^{d_{2}}_{u_{2}}
Y_{t_{1},t_{2}}=\left(I-2u_{1}B_{1}+B_{1}^{2}\right)^{d_{1}} \circ
\left(I-2u_{2}B_{2}+B_{2}^{2}\right)^{d_{2}}
Y_{t_{1},t_{2}}=\varepsilon_{t_{1},t_{2}}, 
%\quad (t_{1},t_{2})\in
%\mathbb{Z}^{2}, 
\end{equation}
  where $\nabla^{d_{i}}_{u_{i}},$ $i=1,2,$ is given by
equation (\ref{fog}), with  $B_{i},$ $i=1,2,$ denoting the
backward-shift operator for each spatial coordinate, i.e.
$B_{1}Y_{t_{1},t_{2}}=Y_{t_{1}-1,t_{2}}$, and
$B_{2}Y_{t_{1},t_{2}}=Y_{t_{1},t_{2}-1}.$ Here,
$\varepsilon_{t_{1},t_{2}},$ $(t_{1},t_{2})\in \mathbb{Z}^{2},$ is
a zero-mean white noise field with the common variance
$E[\varepsilon_{t_{1},t_{2}}^2]=\sigma^{2}_{\varepsilon}.$  The random field $Y$ is called a
spatial Gegenbauer white noise in \citet{Espejo:2014}.

By  equation
(\ref{eqGegenbauer}) the Gegenbauer random field $Y$ can be defined
in terms of the inverse of the operator $\nabla^{d_{1}}_{u_{1}}\circ
\nabla^{d_{2}}_{u_{2}}$ expanded in a Gegenbauer polynomial series
as follows 
\begin{align}
   Y_{t_{1},t_{2}}
        &= \nabla^{-d_{2}}_{u_{2}} \circ \nabla^{-d_{1}}_{u_{1}} \varepsilon_{t_{1},t_{2}}
        = \sum_{n_{1}=0}^{\infty} \sum_{n_{2}=0}^{\infty} \, C_{n_{1}}^{(d_{1})}(u_{1}) \, C_{n_{2}}^{(d_{2})}(u_{2}) B_{1}^{n_{1}} B_{2}^{n_{2}}
        \varepsilon_{t_{1},t_{2}} \nonumber\\
        &= \sum_{n_{1}=0}^{\infty} \sum_{n_{2}=0}^{\infty} \, C_{n_{1}}^{(d_{1})}(u_{1}) \,
        C_{n_{2}}^{(d_{2})}(u_{2}) \varepsilon_{t_{1}-n_{1},t_{2}-n_{2}}, \label{gegenb}
        %\label{Gengebauerprocess}
\end{align}
where $d_i \neq 0,$ $i=1,2,$ and $C_{n}^{(d)}(u)$ is the Gegenbauer polynomial given by \[C_{n}^{(d)}(u)=\sum_{k=0}^{[{n}/{2}]} (-1)^{k} \frac{(2u)^{n-2k}\Gamma(d-k+n)}{k!(n-2k)!\Gamma(d)}.\]
The generating function for the Gegenbauer polynomials is given by 
\[\sum_{n=0}^\infty C_{n}^{(d)}(u)b^n=\left(1-2u b+b^{2}\right)^{-d},\quad |b|<1, 
\]
which explains the expansion of  the inverse operator in (\ref{gegenb}). 

In general,  a  random field $Y$ is called invertible if the white noise $\varepsilon_{t_1,t_2},$  $(t_{1},t_{2})\in \mathbb{Z}^{2},$ can be expressed as the convergent sum 
\[\varepsilon_{t_1,t_2}= \sum_{n_{1}=0}^{\infty} \sum_{n_{2}=0}^{\infty}  b_{n_1,n_2}Y_{t_{1}-n_{1},t_{2}-n_{2}},\]
where $\sum_{n_{1}=0}^{\infty} \sum_{n_{2}=0}^{\infty} | b_{n_1,n_2}|<\infty.$

The Gegenbauer random field  has the following property, see \citet{Espejo:2014}.

\begin{proposition} \label{proposition} If   $0<d_{i}< \frac{1}{2}$ and $|u_{i}|<1,$ $i=1,2,$ then
$Y$ is a stationary invertible long range dependent random field. 
\end{proposition}

The  spectral density of a  stationary Gegenbauer random field 
is given by \citet{Chung:1996a,Chung:1996b}, and \citet{Hsu:2009}, for the one-parameter case, and
\citet{Espejo:2014}, for the two-parameter case:
\begin{align}
    f(\boldsymbol{\lambda},\boldsymbol{\theta})
    &= \frac{\sigma^{2}_{\varepsilon}}{(2\pi)^{2}} \,
    \left| 1-2u_{1}e^{-i\lambda_{1}}+e^{-2i\lambda_{1}} \right|^{-2d_{1}}
    \left| 1-2u_{2}e^{-i\lambda_{2}}+e^{-2i\lambda_{2}} \right|^{-2d_{2}} \nonumber\\
    &= \frac{\sigma^{2}_{\varepsilon}}{(2\pi)^{2}} \,
    \left\{\left|2\cos\lambda_{1}-2u_{1}\right|\right\}^{-2d_{1}} \,
    \left\{\left|2\cos\lambda_{2}-2u_{2}\right|\right\}^{-2d_{2}},
     \label{spectral_density}
\end{align}
  where $\boldsymbol{\theta}=(\boldsymbol{u},\boldsymbol{d})=(u_{1},u_{2},d_{1},d_{2}) \in \Theta=(-1,1)^{2}\times(0,1/2)^{2},$ $u_{i}=\cos \nu_{i}$, and $-\pi \leq \lambda_{i}
\leq \pi,$ $i=1,2.$ Using the spectral density function
(\ref{spectral_density}), one can compute the auto-covariance
function of $Y$ as follows:
\[ \gamma(j_{1},j_{2},\boldsymbol{\theta})= \frac{\sigma^{2}_{\varepsilon}}{4\pi}
    \prod_{i=1}^{2} \Gamma(1-2d_{i}) [2\sin(\nu_{i})]^{\frac{1}{2}-2d_{1}}
    \left[ P_{j_{i}-\frac{1}{2}}^{2d_{i}-\frac{1}{2}} (u_{i}) + (-1)^{j_{i}} P_{j_{i}-\frac{1}{2}}^{2d_{i}-\frac{1}{2}} (-u_{i}) \right],
\]
where $P_{a}^{b}(z)$ is the associated Legendre function of the first kind, consult \S 8 in \citet{Abra:1972}.

  From \citet{Chung:1996a,Chung:1996b}, \citet{Gray:1989} and
\citet{Gradshteyn:1980},  the following asymptotic
approximation of the autocovariance function can be obtained
\begin{align}
    \gamma(j_{1},j_{2},\boldsymbol{\theta})&= \prod_{i=1}^{2} \frac{2^{1-2d_{i}}\sigma^{2}_{\epsilon}}{\pi\sin^{2d_{i}}(\nu_{i})}
      \sin(d_{i}\pi) \Gamma(1-2d_{i}) \cos(j_{i} \nu_{i})\frac{\Gamma(j_{i}+2d_{i})}{\Gamma(j_{i}+1)} [1+\mathcal{O}(j_{i}^{-1})]. \nonumber
\end{align}

The random field $Y$ is long range dependent as its auto-covariance function satisfies the condition
$\sum_{(j_{1},j_{2})\in\mathbb{Z}^{2}}|\gamma(j_{1},j_{2},\boldsymbol{\theta})|= +\infty.$
A detailed discussion on relations between local specifications
of  spectral functions and the tail behaviour
of auto-covariance functions of long range dependent random fields can be found in \cite{LeoOle:2013}.

Figure~\ref{fig2} gives an example of the spectral density and the auto-covariance
function of the Gegenbauer random field for the values of the parameters  $u_1=0.4,$ $u_2= 0.3,$ $d_1= 0.2,$ and $d_2= 0.3.$\vspace{-0.7cm}
\begin{figure}[h]
  \begin{center}
\begin{minipage}{5cm}
\includegraphics[width=6cm, trim=30 0 0 80, clip=true]{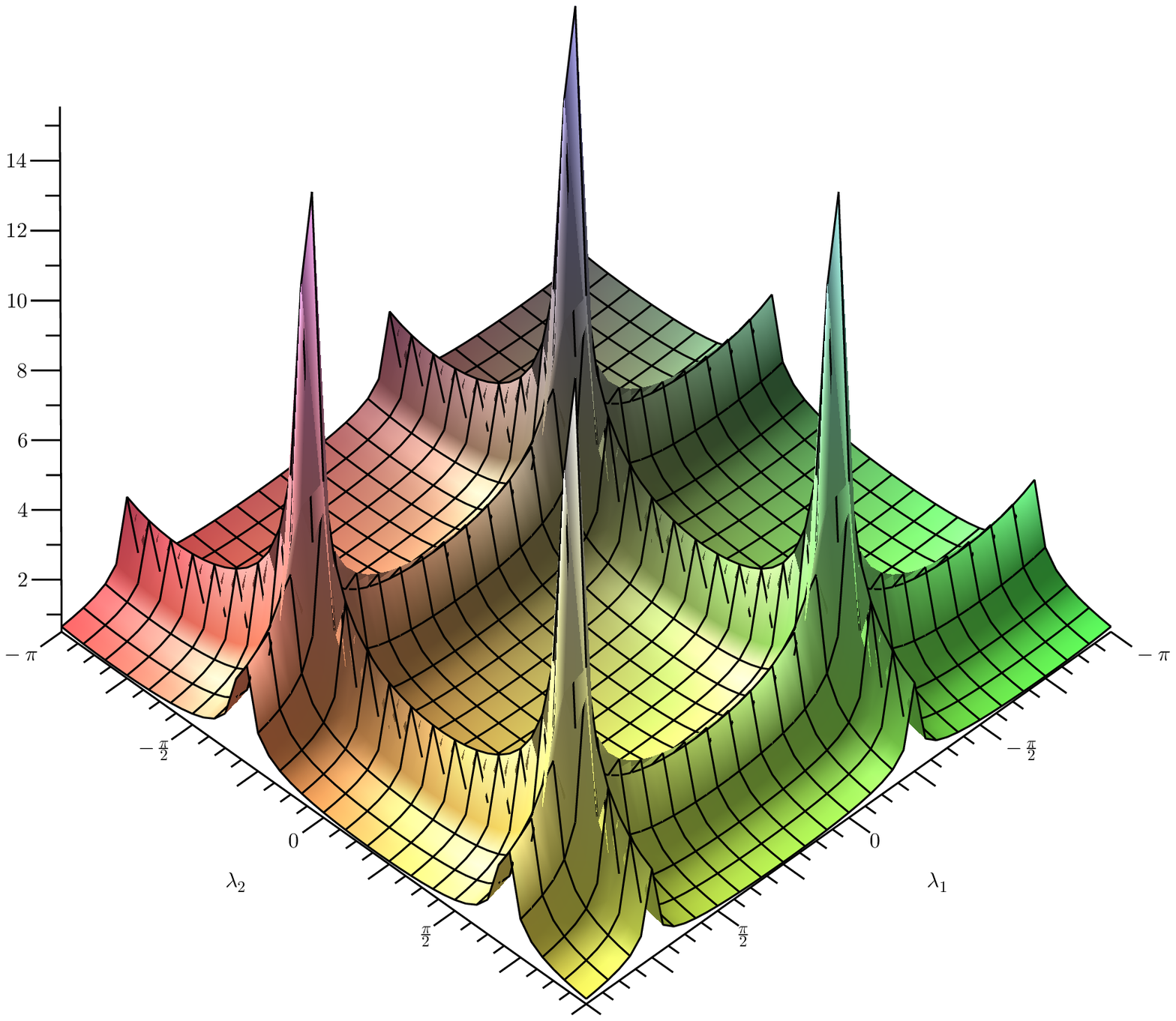} 
 \end{minipage}\quad\quad
\begin{minipage}{5cm}
\includegraphics[width=6cm, trim=30 0 0 80, clip=true]{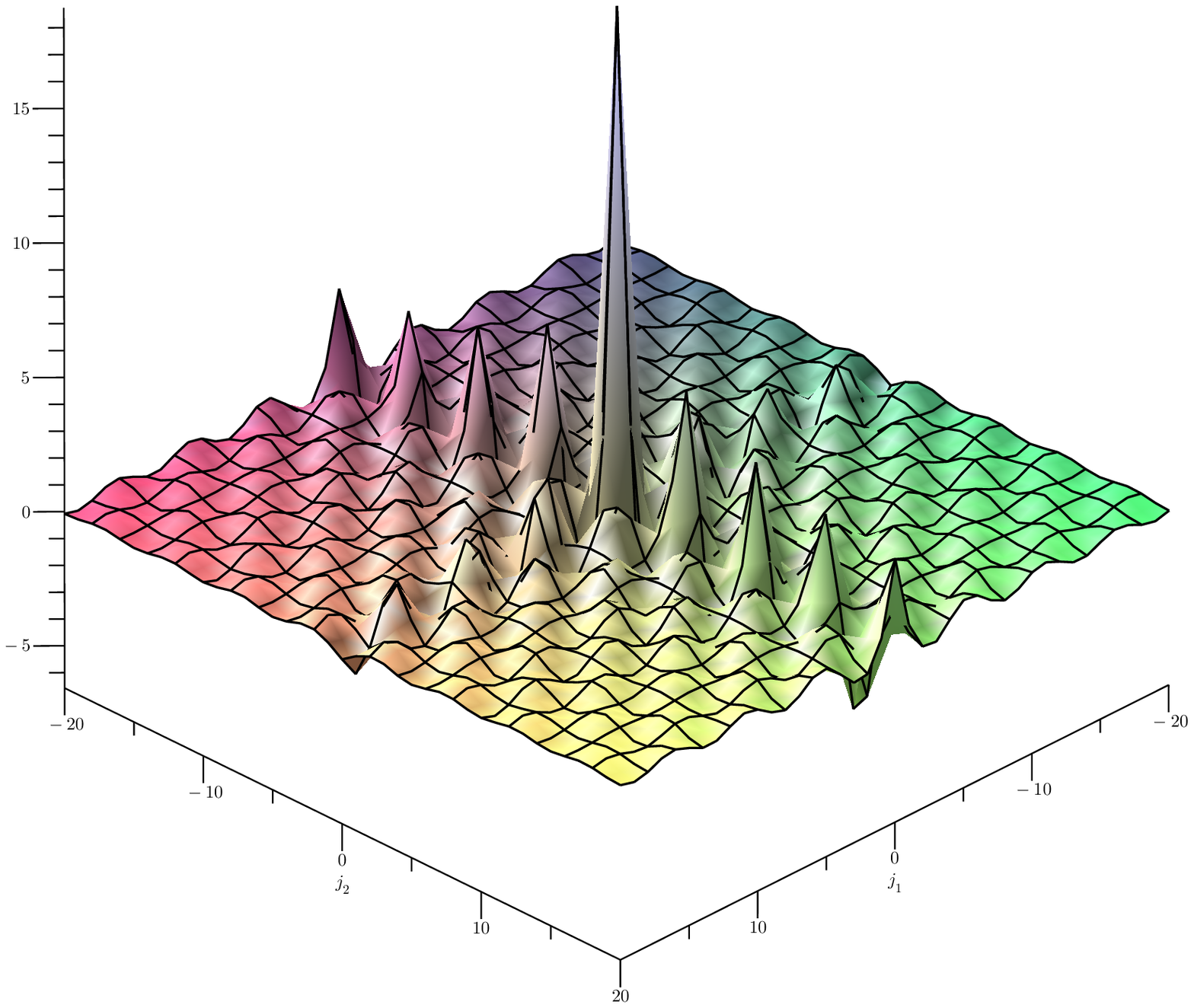} \end{minipage}
\vspace{-0.5cm}\caption{Spectral density and auto-covariance
function for  $\boldsymbol{u}=(0.4,0.3)$ and $\boldsymbol{d}= (0.2, 0.3)$\vspace{-1cm}} \label{fig2}
  \end{center}
\end{figure}
\section{Asymptotic properties of MCEs}
Very detailed discussions about estimation of parameters of seasonal/cyclical long memory time series were given in \cite{Arteche:2000},\cite{Giraitis:2001}, \cite{Ferrara:2001}, and references therein. It was shown that the case of spectral singularities outside the origin is much more difficult comparing to the situation of a spectral pole at 0. In particular, estimators of parameters $u_i$ and $d_i$ have different rates of convergence and their joint distributions are still unknown
under general conditions.   

\cite{Ferrara:2001} pointed out that, in practice, parameter estimation of seasonal/cyclical long memory data is done in two steps. The first step consists in estimation of singularity locations, which does not represent a difficult issue.  Then,  the obtained values of location parameters are used in estimators of the long memory parameters. It has been shown that this 2-steps method provides quasi-similar results to simultaneous procedures. Similarly, the results of this paper may also be used in a hierarchical modeling framework. Namely, locations of the singularities can be included in the first level of the hierarchical procedure.  Then, on the second step of the hierarchical procedure, the long-range dependence parameters can be estimated by the MCE method, conditioning on the locations obtained on the first step.  It also allows to construct a suitable weight function before applying the MCE methodology.  Therefore, in this paper we only concentrate our attention on long-range dependence  parameters.

Suppose  that  the conditions imposed
 in Proposition~\ref{proposition} to ensure stationarity, invertibility  and long-range  dependence hold. Assume  the value of the parameter $\boldsymbol{u}=(u_{1},u_{2})$ is known a priori or was estimated on step 1. Then $\boldsymbol{\theta}=(\theta_{1},\theta_{2})=(d_{1},d_{2})\in \Theta=(0,1/2)^{2}$
 is the vector of parameters to estimate of the Gegenbauer random field defined by equation~(\ref{eqGegenbauer}) (see also equation~(\ref{spectral_density}) for the corresponding spectral density). 
 
Let $Y_{t_{1},t_{2}},$ $t_{1},t_{2}=0,...,T,$ be a part of a realization of the Gegenbauer random field.
 
Let  $w(\boldsymbol{\lambda})$,  $\boldsymbol{\lambda} \in [-\pi,\pi]^{2}$, be a nonnegative function. Suppose the condition \textbf{A3} in the Appendix holds. We define
 \begin{equation}\label{eq_sigma_factorization}
 \sigma^{2} (\boldsymbol{\theta})=\displaystyle\int_{[-\pi,\pi]^{2}}
     f(\boldsymbol{\lambda},\boldsymbol{\theta})w(\boldsymbol{\lambda}) \,\, d\boldsymbol{\lambda}
 \end{equation}
   and consider the factorization
 \begin{equation} \label{equation_spectraldensity}
     f(\boldsymbol{\lambda},\boldsymbol{\theta})= \sigma^{2} (\boldsymbol{\theta}) \Psi(\boldsymbol{\lambda}, \boldsymbol{\theta}).
 \end{equation}
For all $\boldsymbol{\theta }\in \Theta$ the function
 $\Psi(\boldsymbol{\lambda},\boldsymbol{\theta})$
 has the following property
 \begin{equation}\label{equation_psi}
 \displaystyle\int_{[-\pi,\pi]^{2}}
      \Psi(\boldsymbol{\lambda},\boldsymbol{\theta}) w(\boldsymbol{\lambda}) \, \, d\boldsymbol{\lambda} =1.
 \end{equation}

Let $K(\boldsymbol{\theta}_{0},\boldsymbol{\theta})$
be a non-random real-valued function, usually referred as the
\emph{contrast function}, given~by
\[
K(\boldsymbol{\theta_{0}},\boldsymbol{\theta}):=\displaystyle\int_{[-\pi,\pi]^{2}}
f(\boldsymbol{\lambda},\boldsymbol{\theta_{0}})
w(\boldsymbol{\lambda}) \log
\frac{\Psi(\boldsymbol{\lambda},\boldsymbol{\theta_{0}})}
{\Psi(\boldsymbol{\lambda},\boldsymbol{\theta})}  \,\,
d\boldsymbol{\lambda},
\]
  and let the \emph{contrast field} be
\[
    U(\boldsymbol{\theta}):=-\displaystyle\int_{[-\pi,\pi]^{2}}
     f(\boldsymbol{\lambda},\boldsymbol{\theta_{0}}) w(\boldsymbol{\lambda})
     \log \Psi(\boldsymbol{\lambda},\boldsymbol{\theta}) \,\, d\boldsymbol{\lambda},
\]
where 
$\boldsymbol{\theta_{0}}=(d_{10},d_{20})$ is the true parameter value. In what follows, $P_{0}$ denotes the probability distribution with the density function $f(\boldsymbol{\lambda},\boldsymbol{\theta}_0).$

The empirical version $\hat{U}_{T}(\boldsymbol{\theta}),$  $T \in
\mathbb{Z},$ $\boldsymbol{\theta} \in \Theta,$ is defined by
\begin{equation}\label{equation_Ut}
    \hat{U}_{T}(\boldsymbol{\theta}):=-\displaystyle\int_{[-\pi,\pi]^{2}}
     I_{T}(\boldsymbol{\lambda}) w(\boldsymbol{\lambda}) \log\Psi(\boldsymbol{\lambda},\boldsymbol{\theta}) \,\, d\boldsymbol{\lambda},
\end{equation}
  where $I_{T}(\boldsymbol{\lambda})$ is the periodogram of the observations $Y_{t_{1},t_{2}},$ $t_{1},t_{2}=0,...,T,$ of the Gegenbauer random field, that is,
\[
I_{T}(\boldsymbol{\lambda}):= \frac{1}{(2\pi T)^{2}} \left|
 \displaystyle\sum_{t_1=0}^T\sum_{t_2=0}^T e^{-i(t_{1}\lambda_{1}+t_{2}\lambda_{2})} Y_{t_{1},t_{2}} \right|^{2}.
 \]
 
 The MCE is defined by the empirical
 contrast field $\hat{U}_{T}(\boldsymbol{\theta})$   and the contrast
 function $K(\boldsymbol{\theta}_{0},\boldsymbol{\theta})$ being
 $K(\boldsymbol{\theta}_{0},\boldsymbol{\theta})\geq 0,$ and having a
 unique minimum at $\boldsymbol{\theta}=\boldsymbol{\theta}_{0}.$

In particular, we choose
\begin{equation}\label{weightfunction}
   w(\boldsymbol{\lambda})= |\lambda_1^2-\nu_1^2|^{a_1}|\lambda_2^2-\nu_2^2|^{a_2}w_0(\boldsymbol{\lambda}),\quad
   \boldsymbol{\lambda}=(\lambda_{1},\lambda_{2})\in [-\pi,\pi]^{2},
   \end{equation}
   where $a_i>1,$ $i=1,2,$ $w_0(\boldsymbol{\lambda})$ is a positive function with continuous second order derivatives on $[-\pi,\pi]^{2}.$ 

Notice that by (\ref{weightfunction})  the weight function $w(\boldsymbol{\lambda})$ is    nonnegative and symmetric about $(0,0).$ Due to the boundedness of $w(\boldsymbol{\lambda})$ the product  $w(\boldsymbol{\lambda}) f(\boldsymbol{\lambda},\boldsymbol{\theta})$ is integrable. Thus,  the choice of  $w(\boldsymbol{\lambda})$ fulfils condition~{\rm \bf A3} in the Appendix.
The developed methodology is readily adjustable to other classes of weight functions.

 Theorems~\ref{Theorem_1} and \ref{Theorem_2} give 
consistency and asymptotic normality results for the MCE.

\begin{theorem} \label{Theorem_1}
Let $Y_{t_{1},t_{2}},\ (t_{1},t_{2})\in \mathbb{Z}^{2},$ be a
stationary Gegenbauer random field which spectral density satisfies equation {\rm(\ref{spectral_density})}.
If $\hat{U}_{T}(\boldsymbol{\theta})$ is the empirical contrast field 
defined by equation~{\rm(\ref{equation_Ut})}, then
\begin{itemize}
\item $Y_{t_{1},t_{2}}$ satisfies the conditions
{\rm \bf A1-A6} in the Appendix;
\item
 the minimum
contrast estimator 
$\hat{\boldsymbol{\theta}}_{T}= (\hat{d}_{1},\hat{d}_{2})= arg \min_{\boldsymbol{\theta}
\in \Theta} \, \hat{U}_{T}(\boldsymbol{\boldsymbol{\theta}})\in \Theta$
 is a consistent estimator of
the parameter vector $\boldsymbol{\theta}.$ That is, there is a convergence in $P_{0}$ probability:
\[\hat{\boldsymbol{\theta}}_{T}\overset{P_{0}}{\longrightarrow}
\boldsymbol{\theta_{0}}, \quad
T\longrightarrow \infty;\]  
\item $\hat{\sigma}^{2}_{T}\overset{P_{0}}{\longrightarrow}
\sigma^{2}(\boldsymbol{\theta_{0}}),\ T\longrightarrow \infty,$
where the variance estimator  $\hat{\sigma}^{2}_{T}$ is given by
$$\hat{\sigma}^{2}_{T}=\int_{[-\pi, \pi]^2} I_{T}(\boldsymbol{\lambda})w(\boldsymbol{\lambda})\,\, d\boldsymbol{\lambda}.$$
\end{itemize}
\end{theorem}

Notice that Assumption {\bf A4} is not required to prove the last two statements in Theorem~\ref{Theorem_1}, but it will be used in the proof of Theorem~\ref{Theorem_2}.

To formulate Theorem~\ref{Theorem_2} we introduce the following notations.
The unbiased estimator of the correlation function
 $\gamma (t_{1},t_{2},\boldsymbol{\theta }),$ 
 $\boldsymbol{t}=(t_{1},t_{2})\in \mathbb{Z}^{2},$ of the Gegenbauer random field $Y_{t_{1},t_{2}}$
is
\begin{equation}
\hat{\gamma}_{T}(\boldsymbol{t})= \frac{1}{(T-|t_{1}|)(T-|t_{2}|)}
\sum_{k=0}^{T-|t_1|}\sum_{l=0}^{T-|t_2|} Y_{k,l} Y_{|t_{1}|+k,|t_{2}|+l}. \nonumber
\end{equation}
Note that all indices of the random field in the sum above are within the set $\{(t_{1},t_{2}):\,t_{1},t_{2}=0,...,T\},$ where the observations are available.

The unbiased  periodogram is given by
\begin{equation} \label{equation_periodogram_estimado}
I_{T}^{*}(\lambda_{1},\lambda_{2})= \frac{1}{(2\pi)^{2}}
\sum_{t_{1}=1-T}^{T-1}\sum_{t_{2}=1-T}^{T-1} e^{-
i(\lambda_{1}t_{1}+\lambda_{2}t_{2})} \hat{\gamma}_{T}(\boldsymbol{t}),
\nonumber
\end{equation}
  and the corresponding empirical \emph{contrast field} is
\[
    \hat{U}_{T}^{*}(\boldsymbol{\theta})=-\displaystyle\int_{[-\pi, \pi]^2}
     I_{T}^{*}(\boldsymbol{\lambda}) w(\boldsymbol{\lambda}) \log \Psi(\boldsymbol{\lambda},\boldsymbol{\theta}) \,\, d\boldsymbol{\lambda}.
\]
We also define $\hat{\sigma}^{2*}_{T}=\int_{[-\pi, \pi]^2}
I_{T}^{*}(\boldsymbol{\lambda})w(\boldsymbol{\lambda})\,\,
 d\boldsymbol{\lambda}$ and the associated adjusted MCE
\begin{equation} \label{equation_theta_estimado2}
    \hat{\boldsymbol{\theta}}^{*}_{T}= (\hat{d}^{*}_{1}, \hat{d}^{*}_{2})= \mbox{arg} \min_{\boldsymbol{\theta}
     \in \Theta} \, \hat{U}_{T}^{*}(\boldsymbol{\theta}).
\end{equation}

\begin{theorem}\label{Theorem_2}
If $Y_{t_{1},t_{2}},\ (t_{1},t_{2})\in \mathbb{Z}^{2},$ is a
stationary Gegenbauer random field which spectral density satisfies equation {\rm(\ref{spectral_density})} with $(d_1,d_2)\in (0,1/4)^2,$ then
\begin{itemize}
\item $Y_{t_{1},t_{2}}$ satisfies the conditions
{\rm \bf A1-A9} in the Appendix;
\item the adjusted MCE defined by
{\rm(\ref{equation_theta_estimado2})} is asymptotically normal. That is,
$$T(\hat{\boldsymbol{\theta}}^{*}_{T}-\boldsymbol{\theta}_{0})\stackrel{D}{\longrightarrow}\mathcal{N}_{2}
(0,\mathbf{S}^{-1}(\boldsymbol{\theta}_{0})
A(\boldsymbol{\theta}_{0})\mathbf{S}^{-1}(\boldsymbol{\theta}_{0})),\quad  T\longrightarrow \infty,$$
where the entries of the matrices $\mathbf{S}(\boldsymbol{\theta})=(s_{ij}(\boldsymbol{\theta}))$      and $\mathbf{A}(\boldsymbol{\theta})=(a_{ij}(\boldsymbol{\theta}))$  are  
\begin{eqnarray}
{}\hspace{-1cm}s_{ij}(\boldsymbol{\theta})&=& \displaystyle\int_{[-\pi,\pi]^{2}}
f(\boldsymbol{\lambda},\boldsymbol{\theta}) w(\boldsymbol{\lambda})
\frac{\partial^{2}}{\partial \theta_{i}\partial \theta_{j}}
\log \Psi(\boldsymbol{\lambda},\boldsymbol{\theta}) \,\, d\boldsymbol{\lambda}=\sigma^{2}(\boldsymbol{\theta})
\displaystyle\int_{[-\pi,\pi]^{2}} w(\boldsymbol{\lambda}) \nonumber \\
& & \times\left[
\frac{\partial^{2}}{\partial \theta_{i}\partial \theta_{j}}
\Psi(\boldsymbol{\lambda},\boldsymbol{\theta}) -
\frac{1}{\Psi(\boldsymbol{\lambda},\boldsymbol{\theta})}
  \frac{\partial}{\partial \theta_{i}} \Psi(\boldsymbol{\lambda},\boldsymbol{\theta})
  \frac{\partial}{\partial \theta_{j}} \Psi(\boldsymbol{\lambda},\boldsymbol{\theta})
\right] \, d\boldsymbol{\lambda}, \label{sij}
\end{eqnarray}
\begin{eqnarray}
a_{ij}(\boldsymbol{\theta})&=& 8\pi^{2}
\displaystyle\int_{[-\pi,\pi]^{2}}
f^{2}(\boldsymbol{\lambda},\boldsymbol{\theta})
w^{2}(\boldsymbol{\lambda}) \frac{\partial}{\partial \theta_{i}}
\log \left(\Psi(\boldsymbol{\lambda},\boldsymbol{\theta}) \right)
\frac{\partial}{\partial \theta_{j}} \log
\left(\Psi(\boldsymbol{\lambda},\boldsymbol{\theta}) \right)
\,\, d\boldsymbol{\lambda} \nonumber \\
&=&  8\pi^{2} \sigma^{4}(\boldsymbol{\theta})
\displaystyle\int_{[-\pi,\pi]^{2}} w^{2}(\boldsymbol{\lambda})
\frac{\partial}{\partial \theta_{i}}
\Psi(\boldsymbol{\lambda},\boldsymbol{\theta})
 \frac{\partial}{\partial \theta_{j}} \Psi(\boldsymbol{\lambda},\boldsymbol{\theta})
 \,\, d\boldsymbol{\lambda}. \label{aij}
\end{eqnarray}
\end{itemize}
\end{theorem}

To avoid the edge effect \citet{Anh:2004} employed the modified periodogram approach suggested by \citet{Guyon:1982}. We use their assumptions in Theorem~{\rm\ref{Theorem_2}}. Note that some authors pointed few problems in using $I_{T}^{*},$ see \citet{Sanz:2009}, \citet{Yao:2006}, and references therein.
Various other modifications, for example, typed variograms, smoothed variograms, kernel estimators,   to reduce the edge effect have been
proposed. It would be interesting to prove analogous of the results by \citet{Anh:2004} and Theorem~{\rm\ref{Theorem_2}} for these modifications, too.   
However, it is beyond the scope of this paper. Moreover, it remains as an open problem whether the edge-effect modification is essential for the asymptotic normality or not, see \citet{Yao:2006}.

\section{Proofs}
To prove the theorems we will use the following differentiability lemma.
\begin{lemma}{\rm[\citet[Theorem 11.5]{Schilling:2005}]}\label{lem1}
Let $(X,\cal{F},\mu)$ be a measurable space, $A\in\cal{F},$ and $\Theta\subset\mathbb{R}$ be an open set. Suppose the function $F: X\times \Theta\to\mathbb{R}$ satisfies the following conditions:
\begin{enumerate}
\item For all $\theta\in \Theta:$ $F(\cdot,\theta)\in L_1(A);$
\item For almost all $x\in A$  the derivative $\frac{\partial F(x,\cdot)}{\partial \theta}$ exists for all $\theta\in \Theta;$
\item     There is an integrable function $g:X\to \mathbb{R}$ such that $\left|\frac{\partial F(x,\theta)}{\partial\theta}\right|\le g(x)$  for almost all $x\in A.$
\end{enumerate}
Then there exists 
$\frac{\partial}{\partial \theta}\int_A F(x,\theta)d\mu(x)=\int_A \frac{\partial F(x,\theta)}{\partial \theta}d\mu(x).$
\end{lemma}

\begin{lemma}\label{lem2} The function $\sigma^{2}(\boldsymbol{\theta
})$ is bounded and separated from zero on $\Theta.$ Moreover,  its first and second order derivatives are bounded on $\Theta$ and can be computed by
\begin{eqnarray}\label{dtheta}\frac{\partial}{\partial
  \theta_{i}}\sigma^{2}(\boldsymbol{\theta})&=&
\int_{[-\pi,\pi]^{2}}
w(\boldsymbol{\lambda}) \frac{\partial}{\partial
        \theta_{i}} f(\boldsymbol{\lambda},\boldsymbol{\theta}) \, d\boldsymbol{\lambda}\nonumber\\ &=&-2\int_{[-\pi,\pi]^{2}}
        \log\left|2\cos\lambda_{i}-2u_{i}\right|\,w(\boldsymbol{\lambda})  f(\boldsymbol{\lambda},\boldsymbol{\theta}) \, d\boldsymbol{\lambda},\end{eqnarray}
\begin{eqnarray}\frac{\partial^{2}}{\partial
                    \theta_{j}\partial \theta_{i}}\sigma^{2}(\boldsymbol{\theta})&=&
                 \displaystyle\int_{[-\pi,\pi]^{2}}
              w(\boldsymbol{\lambda}) \frac{\partial^{2}}{\partial
                  \theta_{j}\partial \theta_{i}} f(\boldsymbol{\lambda},\boldsymbol{\theta}) \, d\boldsymbol{\lambda}= 4\int_{[-\pi,\pi]^{2}}
                          \log\left|2\cos\lambda_{i}-2u_{i}\right|\nonumber\\
& & \times \log\left|2\cos\lambda_{j}-2u_{j}\right|\,w(\boldsymbol{\lambda})  f(\boldsymbol{\lambda},\boldsymbol{\theta}) \, d\boldsymbol{\lambda},\label{d2theta}    \end{eqnarray}
where $i,j=1,2.$ 
\end{lemma}
\begin{proof}
 By the choice (\ref{weightfunction}) of the weight function we obtain
\begin{equation}\label{fw1} 
\sup_{[-\pi,\pi]^{2}\times \Theta}f(\boldsymbol{\lambda },\boldsymbol{\theta
})w(\boldsymbol{\lambda })<+\infty\end{equation} 
and 
\[\sup_{\theta\in \Theta}\sigma^{2}(\boldsymbol{\theta
})=\int_{[-\pi,\pi]^{2}}\sup_{\theta\in \Theta}f(\boldsymbol{\lambda },\boldsymbol{\theta
})w(\boldsymbol{\lambda })d\boldsymbol{\lambda }\leq 4\pi^{2} \sup_{[-\pi,\pi]^{2}\times \Theta}f(\boldsymbol{\lambda },\boldsymbol{\theta
})w(\boldsymbol{\lambda })<+\infty. 
\] Hence, $\sigma^{2}(\boldsymbol{\theta
})$ is bounded.

Note that   $\sup_{[-\pi,\pi]^{2}\times \Theta}\left\{\left|2\cos\lambda_{1}-2u_{1}\right|\right\}^{2d_{1}} \,
    \left\{\left|2\cos\lambda_{2}-2u_{2}\right|\right\}^{2d_{2}}<+\infty.
$
Also, by the choice of the weight function, there exists $\delta>0$ and a set $A_0\subset [-\pi,\pi]^{2}$ of non-zero Lebesgue measure  such that $w(\boldsymbol{\lambda })>\delta$ for all $\boldsymbol{\lambda }\in A_0.$ Therefore, 
 \begin{align}\inf_{\theta\in \Theta}\sigma^{2}(\boldsymbol{\theta
})&=\inf_{\theta\in \Theta}\int_{[-\pi,\pi]^{2}}f(\boldsymbol{\lambda },\boldsymbol{\theta
})w(\boldsymbol{\lambda })d\boldsymbol{\lambda }\nonumber\\
&\geq \frac{\delta\, \boldsymbol{\lambda }(A_0)}{\sup_{[-\pi,\pi]^{2}\times \Theta}\left\{\left|2\cos\lambda_{1}-2u_{1}\right|\right\}^{2d_{1}} \,
    \left\{\left|2\cos\lambda_{2}-2u_{2}\right|\right\}^{2d_{2}}}>0,\nonumber 
 \end{align}
which means that $\sigma^{2}(\boldsymbol{\theta
})$ is separated from zero on $\Theta.$

Now, to study $\frac{\partial}{\partial
\theta_{i}}\sigma^{2}(\boldsymbol{\theta})$ we compute  $\frac{\partial}{\partial \theta_{i}} f(\boldsymbol{\lambda},\boldsymbol{\theta
}), $
  $i=1,2:$ 
        \begin{align}
        \frac{\partial}{\partial \theta_{i}} f(\boldsymbol{\lambda},\boldsymbol{\theta
        }) &=
        \frac{\sigma^{2}_{\varepsilon}}{(2\pi)^{2}} \left\{\left|2\cos\lambda_{j}-2u_{j}\right|\right\}^{-2d_{j}}
        \frac{\partial}{\partial d_{i}} \left\{\left|2\cos\lambda_{i}-2u_{i}\right|\right\}^{-2d_{i}} \nonumber\\
        &= \log\left(\left(2\cos\lambda_{i}-2u_{i}\right)^{-2}\right) \frac{\sigma^{2}_{\varepsilon}}{(2\pi)^{2}}
        \left[2(\cos\lambda_{i}-u_{i}) \right]^{-2d_{i}} \nonumber\\
        &\quad\times\left[2(\cos\lambda_{j}-u_{j}) \right]^{-2d_{j}} = -2\log\left|2\cos\lambda_{i}-2u_{i}\right| f(\boldsymbol{\lambda},\boldsymbol{\theta}), \label{deriv_f_d}
        \end{align}
where   $j\not= i$ and     $j=1,2.$  

   Using  (\ref{weightfunction}) and (\ref{deriv_f_d})  we conclude that 
 \begin{equation}\label{dsup}\sup_{[-\pi,\pi]^{2}\times \Theta}\left|w(\boldsymbol{\lambda}) \frac{\partial}{\partial
  \theta_{i}} f(\boldsymbol{\lambda},\boldsymbol{\theta})\right|<+\infty.\end{equation}
  Thus, by (\ref{eq_sigma_factorization}) and Lemma~\ref{lem1}  there exists \[\frac{\partial}{\partial
  \theta_{i}}\sigma^{2}(\boldsymbol{\theta})=
   \displaystyle\int_{[-\pi,\pi]^{2}}
w(\boldsymbol{\lambda}) \frac{\partial}{\partial
        \theta_{i}} f(\boldsymbol{\lambda},\boldsymbol{\theta}) \,\, d\boldsymbol{\lambda}\]
         and
\[\sup_{\Theta}\left|\frac{\partial}{\partial
  \theta_{i}}\sigma^{2}(\boldsymbol{\theta})\right|\leq 4\pi^2\sup_{[-\pi,\pi]^{2}\times \Theta}\left|w(\boldsymbol{\lambda}) \frac{\partial}{\partial
    \theta_{i}} f(\boldsymbol{\lambda},\boldsymbol{\theta})\right|<+\infty.\]
    
    It is not difficult to find 
    $\frac{\partial^{2}} {\partial \theta_{i} \partial \theta_{j}}
    f(\boldsymbol{\lambda},\boldsymbol{\theta}).$
    By (\ref{deriv_f_d}), for $i,j=1,2,$ $i\neq j,$ the second derivatives of $f$  are given by
    \begin{equation}\label{eqsdds3}
    \frac{\partial^{2}}{\partial \theta^{2}_{i}} f(\boldsymbol{\lambda},\boldsymbol{\theta
    }) = -2\log\left|2\cos\lambda_{i}-2u_{i}\right|\frac{\partial}{\partial \theta_{i}}
    f(\boldsymbol{\lambda},\boldsymbol{\theta}) = 4 \left(\log\left|2\cos\lambda_{i}-2u_{i}\right|\right)^{2} f(\boldsymbol{\lambda},\boldsymbol{\theta}),
    \end{equation}
    \begin{equation}\label{eqsdds4}
    \frac{\partial^{2}}{\partial \theta_{j} \partial \theta_{i}} f(\boldsymbol{\lambda},\boldsymbol{\theta
    })
    = 4\log\left|2\cos\lambda_{i}-2u_{i}\right|\cdot \log\left|2\cos\lambda_{j}-2u_{j}\right| f(\boldsymbol{\lambda},\boldsymbol{\theta}).
    \end{equation}
    
    It follows from (\ref{eqsdds3}), (\ref{eqsdds4}), and (\ref{weightfunction})  that  
    \begin{equation}\label{df2}\sup_{[-\pi,\pi]^{2}\times \Theta}\left|w(\boldsymbol{\lambda}) \frac{\partial^{2}}{\partial
        \theta_{j}\partial \theta_{i}} f(\boldsymbol{\lambda},\boldsymbol{\theta})\right|<+\infty,\quad i,j=1,2.\end{equation}
Finally, by (\ref{dtheta}) and Lemma~\ref{lem1}  there exists \[\frac{\partial^{2}}{\partial
            \theta_{j}\partial \theta_{i}}\sigma^{2}(\boldsymbol{\theta})=
         \displaystyle\int_{[-\pi,\pi]^{2}}
      w(\boldsymbol{\lambda}) \frac{\partial^{2}}{\partial
          \theta_{j}\partial \theta_{i}} f(\boldsymbol{\lambda},\boldsymbol{\theta}) \,\, d\boldsymbol{\lambda}\] and
    \begin{equation}\label{d2sigma}\sup_{\Theta}\left|\frac{\partial^{2}}{\partial
          \theta_{j}\partial \theta_{i}}\sigma^{2}(\boldsymbol{\theta})\right|\leq 4\pi^2\sup_{[-\pi,\pi]^{2}\times \Theta}\left|w(\boldsymbol{\lambda}) \frac{\partial^{2}}{\partial
              \theta_{j}\partial \theta_{i}} f(\boldsymbol{\lambda},\boldsymbol{\theta})\right|<+\infty.\vspace{-3mm}\end{equation}
  \hfill $\vspace{-3mm}\blacksquare$
\end{proof}

\noindent\begin{proof}\textit{of Theorem} \ref{Theorem_1}.
We will prove that the conditions \textbf{A1-A6} in the Appendix are satisfied. Therefore, we will be able to apply Theorem 3 by \citet{Anh:2004} and obtain the statement of Theorem~\ref{Theorem_1}. 

The condition \textbf{A1} holds, since $\boldsymbol{\theta}_{0}$ belongs to the parameter
  space $\Theta = \left(0,\frac{1}{2}\right)^{2}$
  which is an interior of the compact set $\left[0,\frac{1}{2}\right]^{2}$.

It follows from representation (\ref{spectral_density}) of the spectral density that  $f(\boldsymbol{\lambda
},\boldsymbol{\theta}_{1})\neq f(\boldsymbol{\lambda
},\boldsymbol{\theta}_{2}),$  for $\boldsymbol{\theta}_{1}\neq
\boldsymbol{\theta}_{2}.$ Thus, the condition {\bfseries A2} is satisfied.  

The class of non-negative weight functions $w(\boldsymbol{\lambda})$ defined by (\ref{weightfunction}) consists of  symmetric functions. Note that $$|\cos(\lambda_i)-\cos(\nu_i)|=2\left|\sin\left(\frac{\lambda_i+\nu_i}{2}\right)\sin\left(\frac{\lambda_i-\nu_i}{2}\right)\right|\sim C\,|\lambda_i^2-\nu_i^2|,$$ 
when $\lambda_i\to\pm \nu_i.$ 
Thus, by (\ref{weightfunction}) and representation (\ref{spectral_density})  of the spectral density
 we get  $w(\boldsymbol{\lambda})f(\boldsymbol{\lambda
},\boldsymbol{\theta})\in
    L_{1}([-\pi,\pi]^{2})$ for all $\boldsymbol{\theta}.$

To verify \textbf{A4}, that is, to prove  $$\nabla_{\boldsymbol{\theta}} \displaystyle\int_{[-\pi,\pi]^{2}}
     \Psi(\boldsymbol{\lambda},\boldsymbol{\theta}) w(\boldsymbol{\lambda})\,\, d\boldsymbol{\lambda} =
            \displaystyle\int_{[-\pi,\pi]^{2}} \nabla_{\boldsymbol{\theta}}
     \Psi(\boldsymbol{\lambda},\boldsymbol{\theta}) w(\boldsymbol{\lambda})\,\, d\boldsymbol{\lambda} =
     0,$$
we find
\begin{equation}\label{S2}w(\boldsymbol{\lambda}) \frac{\partial}{\partial
\theta_{i}}\Psi(\boldsymbol{\lambda},\boldsymbol{\theta}) 
=\underbrace{\frac{w(\boldsymbol{\lambda})}{\sigma^{2}(\boldsymbol{\theta})}  \left[\frac{\partial}{\partial
\theta_{i}} f(\boldsymbol{\lambda},\boldsymbol{\theta}) \right]}_{S_{1}(\boldsymbol{\lambda},\boldsymbol{\theta})}
 - \underbrace{\frac{w(\boldsymbol{\lambda})}{\sigma^{4}(\boldsymbol{\theta})}\left[\frac{\partial}{\partial
\theta_{i}} \sigma^{2}(\boldsymbol{\theta}) \right]
f(\boldsymbol{\lambda},\boldsymbol{\theta})}_{S_{2}(\boldsymbol{\lambda},\boldsymbol{\theta})}
\end{equation}
and apply Lemma 1.

The same symbol $C$ is used for different nonessential constants appearing in the calculations below.

By (\ref{spectral_density}) and the choice of the weight function we obtain 
\begin{equation}\label{fw2} \sup_{[-\pi,\pi]^{2}\times \Theta}\big|\log\left|2\cos\lambda_{i}-2u_{i}\right|\big|\,f(\boldsymbol{\lambda },\boldsymbol{\theta
})w(\boldsymbol{\lambda })<+\infty.\end{equation}

Therefore, by  equation  (\ref{deriv_f_d}) and Lemma~\ref{lem2}:
\begin{equation}\label{S1up}
|S_{1}(\boldsymbol{\lambda},\boldsymbol{\theta})| \leq \frac{
C}{\sigma^{2}(\boldsymbol{\theta})}\leq \frac{
C}{\min_{\Theta}\sigma^{2}(\boldsymbol{\theta})}  \in L_{1}
([-\pi,\pi]^{2}).\end{equation}

Now, by (\ref{fw1}), (\ref{S2}) and Lemma~\ref{lem2} we can estimate $S_{2}(\boldsymbol{\lambda},\boldsymbol{\theta})$ as
\begin{equation}\label{S2up} \left|S_{2}(\boldsymbol{\lambda},\boldsymbol{\theta})\right| =\left| w(\boldsymbol{\lambda})
\frac{\left[\frac{\partial}{\partial
\theta_{i}}\sigma^{2}(\boldsymbol{\theta}) \right]
f(\boldsymbol{\lambda},\boldsymbol{\theta})}
{\sigma^{4}(\boldsymbol{\theta})}\right|\leq C\, 
\frac{w(\boldsymbol{\lambda})f(\boldsymbol{\lambda},\boldsymbol{\theta})}
{\min_{\Theta}\sigma^{4}(\boldsymbol{\theta})}\le C.\end{equation} 

Finally, \textbf{A4} follows from (\ref{S1up}), (\ref{S2up}), and Lemma~\ref{lem1} with $g(x)=C.$

Note that $L_{1}([-\pi,\pi]^{2})\cap L_{2}([-\pi,\pi]^{2})= L_{2}([-\pi,\pi]^{2}).$ To verify the condition \textbf{A5} for the weight function $w(\cdot)$  we have to show that 
$f(\boldsymbol{\lambda },\boldsymbol{\theta }_{0})w(\boldsymbol{\lambda
    })\cdot$ $\log\Psi (\boldsymbol{\lambda
    }, \boldsymbol{\theta
    })\in  L_{2}([-\pi,\pi]^{2}),$ for all $\boldsymbol{\theta} \in \Theta .$
By (\ref{spectral_density}) and (\ref{weightfunction}) the product $f(\boldsymbol{\lambda},\boldsymbol{\theta}_{0})\cdot$ $ w(\boldsymbol{\lambda})
\log \Psi(\boldsymbol{\lambda},\boldsymbol{\theta})$ is bounded for all $\boldsymbol{\lambda}$ except $\{\boldsymbol{\lambda}: \lambda_i=\pm \nu_i,\ i=1,2\}.$ Let $\tilde{d}_i=\max ({d}_i,{d}_{i0}).$ Then, for
$\lambda_{i}\to \pm\nu_{i},$ $i=1,2:$ 
$$[f(\boldsymbol{\lambda},\boldsymbol{\theta}_0)
w(\boldsymbol{\lambda}) \log
\Psi(\boldsymbol{\lambda},\boldsymbol{\theta})]^{2}\leq
C\prod_{i=1}^{2}|{\lambda_i}\pm {\nu_i}|^{2a_i-4\tilde{d}_i}\log|\lambda_{i}\pm \nu_{i}|\in  L_{1}([-\pi,\pi]^{2}).
$$
Therefore, combining the above results we conclude that \textbf{A5} holds.

 To verify the condition \textbf{A6} we use the following function
$v(\boldsymbol{\lambda})= |\lambda_{1}^2-\nu_{1}^2|^{\beta}|\lambda_{2}^2-\nu_{2}^2|^{\beta},$ $\beta \in (0,1/2).$
Note that       \[\frac{|\lambda_{1}^2-\nu_{1}^2|^{\beta}}{|\lambda_{2}^2-\nu_{2}^2|^{-\beta}}\log f(\boldsymbol{\lambda},\boldsymbol{\theta})\sim C \frac{|\lambda_{1}^2-\nu_{1}^2|^{\beta}}{|\lambda_{2}^2-\nu_{2}^2|^{-\beta}}\left(d_1\log|\lambda_1^2-\nu_1^2|+d_2\log|\lambda_1^2-\nu_1^2|\right)\to 0,\]
when $\lambda_i\to\pm \nu_i.$ Thus, by the choice of $v(\cdot)$ and properties of $\sigma(\boldsymbol{\theta})$
     the function \[h(\boldsymbol{\lambda},\boldsymbol{\theta})= v(\boldsymbol{\lambda}) \log\Psi(\boldsymbol{\lambda},\boldsymbol{\theta})=|\lambda_{1}^2-\nu_{1}^2|^{\beta}|\lambda_{2}^2-\nu_{2}^2|^{\beta}\left(\log f(\boldsymbol{\lambda},\boldsymbol{\theta})-2\log \sigma(\boldsymbol{\theta})\right)\]
     is uniformly continuous on $[-\pi,\pi]^{2}\times \Theta .$
     
       Also, it holds
 \[\left|f(\boldsymbol{\lambda },\boldsymbol{\theta }_{0})
     \frac{w(\boldsymbol{\lambda })}{v(\boldsymbol{\lambda })}\right|\le Cv^{-1}(\boldsymbol{\lambda })\in  L_{2}([-\pi,\pi]^{2}).\]

   Since the conditions \textbf{A1}-\textbf{A6} are satisfied 
   Theorem 1 follows from Theorem 3 in \citet{Anh:2004}.\hfill $\blacksquare$
\end{proof}

\begin{proof} \textit{of Theorem} 2.
To prove the asymptotic normality of the MCE in Theorem~\ref{Theorem_2} we will show 
that the conditions \textbf{A7}-\textbf{A9} of the Appendix  hold.

We begin by proving  the condition \textbf{A7}.  First, to verify the twice differentiability of the function
     $\Psi(\boldsymbol{\lambda},\boldsymbol{\boldsymbol{\theta}})$
     on  $\Theta$ we formally compute the second-order derivatives of $\Psi:$
\begin{eqnarray}
 \frac{\partial^{2}} {\partial \theta_{j}\partial \theta_{i}}  \Psi(\boldsymbol{\lambda},\boldsymbol{\boldsymbol{\theta}}) &=&
 \frac{\partial} {\partial \theta_{j}} \left[ \frac{
 \left[\frac{\partial} {\partial \theta_{i}} f(\boldsymbol{\lambda},\boldsymbol{\theta})\right]  \sigma^{2}(\boldsymbol{\theta }) -
 \left[\frac{\partial} {\partial \theta_{i}} \sigma^{2}(\boldsymbol{\theta }) \right]
 f(\boldsymbol{\lambda},\boldsymbol{\theta})}{\sigma^{4}(\boldsymbol{\theta }) }
 \right] \nonumber\\
 &=&\frac{1}{\sigma^{4}(\boldsymbol{\theta })} \left\{
  \left[\frac{\partial^{2}} {\partial \theta_{j}\partial \theta_{i}}
  f(\boldsymbol{\lambda},\boldsymbol{\theta})\right]  \sigma^{2}(\boldsymbol{\theta }) +
 \left[\frac{\partial} {\partial \theta_{j}} \sigma^{2}(\boldsymbol{\theta })\right]
 \left[\frac{\partial} {\partial \theta}_{i} f(\boldsymbol{\lambda},\boldsymbol{\theta})\right] \right\} \nonumber\\
& -&\frac{1}{\sigma^{4}(\boldsymbol{\theta })} \left\{
 \left[\frac{\partial^{2}} {\partial \theta_{j}\partial \theta_{i}}
 \sigma^{2}(\boldsymbol{\theta }) \right] f(\boldsymbol{\lambda},\boldsymbol{\theta})-
 \left[\frac{\partial} {\partial \theta_{i}} \sigma^{2}(\boldsymbol{\theta }) \right]
 \left[\frac{\partial} {\partial \theta_{j}} f(\boldsymbol{\lambda},\boldsymbol{\theta})\right]  \right\}\nonumber \\
& -&\frac{1}{\sigma^{8}(\boldsymbol{\theta })} \left\{
\frac{\partial} {\partial \theta_{j}} \sigma^{4}(\boldsymbol{\theta
})\left( \left[\frac{\partial}{\partial \theta_{i}}
f(\boldsymbol{\lambda},\boldsymbol{\theta})\right]\sigma^{2}(\boldsymbol{\theta
})
 - \left[\frac{\partial} {\partial \theta_{i}}
\sigma^{2}(\boldsymbol{\theta }) \right]
f(\boldsymbol{\lambda},\boldsymbol{\theta}) \right)
\right\}.\nonumber
\end{eqnarray}
 
Note that in Lemma~\ref{lem2} we proved that the  derivatives $\frac{\partial}{\partial \theta_{i}}
\sigma^{2}(\boldsymbol{\theta }),$ $\frac{\partial^{2}} {\partial \theta_{j}\partial \theta_{i}}
 \sigma^{2}(\boldsymbol{\theta }),$ $\frac{\partial}
 {\partial \theta_{i}} f(\boldsymbol{\lambda},\boldsymbol{\theta}),$ and $\frac{\partial^2}
  {\partial \theta_{i}\partial \theta_{j}} f(\boldsymbol{\lambda},\boldsymbol{\theta})$ exist. 
Hence, by the above computations and Lemma~\ref{lem2} the function $\Psi $ is twice
differentiable  on  $\Theta.$

In addition,  by estimates (\ref{dsup}), (\ref{df2}), (\ref{d2sigma}),  Lemma~\ref{lem2}, and the above representation  for $ \frac{\partial^{2}} {\partial \theta_{j}\partial \theta_{i}}  \Psi(\boldsymbol{\lambda},\boldsymbol{\boldsymbol{\theta}})$ the product $w(\boldsymbol{\lambda}) f(\boldsymbol{\lambda},\boldsymbol{\theta}_{0})
           \frac{\partial^{2}}{\partial\theta_{i} \partial \theta_{j}}
                \log\Psi(\boldsymbol{\lambda},\boldsymbol{\theta})$ is bounded on $[-\pi,\pi]^{2}\times \Theta.$ Hence, for $ i,j=1,2,$ $\boldsymbol{\theta }\in \Theta:$
$$w(\boldsymbol{\lambda}) f(\boldsymbol{\lambda},\boldsymbol{\theta}_{0})
           \frac{\partial^{2}}{\partial\theta_{i} \partial \theta_{j}}
                \log\Psi(\boldsymbol{\lambda},\boldsymbol{\theta}) \in L_{1}([-\pi,\pi]^{2})\cap
                L_{2}([-\pi,\pi]^{2}). $$

To prove part 2 of the condition \textbf{A7}, we first note that by (\ref{equation_spectraldensity}) and (\ref{deriv_f_d}):
\[\frac{\partial}{\partial \theta_{i}}
                \log \Psi(\boldsymbol{\lambda},\boldsymbol{\theta}) =-2\log\left|2\cos\lambda_{i}-2u_{i}\right|-\frac{\frac{\partial}{\partial
                \theta_{i}}\sigma^{2}(\boldsymbol{\theta})}
                {\sigma^{4}(\boldsymbol{\theta})}. \]
By Lemma~\ref{lem2} the second term is bounded. Hence, it follows from (\ref{fw1}) and (\ref{fw2}) that the product  $w(\boldsymbol{\lambda})f(\boldsymbol{\lambda},\boldsymbol{\theta}_{0})\frac{\partial}{\partial \theta_{i}}
                \log \Psi(\boldsymbol{\lambda},\boldsymbol{\theta})$ is bounded on $[-\pi,\pi]^{2}\times \Theta,$ that implies
\[w(\boldsymbol{\lambda})f(\boldsymbol{\lambda},\boldsymbol{\theta}_{0})\frac{\partial}{\partial \theta_{i}}
                \log \Psi(\boldsymbol{\lambda},\boldsymbol{\theta}) \in L_{k}([-\pi,\pi]^{2}),\quad k\geq 1,\ \  i=1,2,
                \ \ \boldsymbol{\theta }\in \Theta .\]

To verify the condition {\bf A8} we first check the positive definiteness of the matrices $S(\boldsymbol{\theta })$ and $\mathbf{A}(\boldsymbol{\theta}).$

The entries of  $S(\boldsymbol{\theta })$ can be rewritten as
\[s_{i,j}(\boldsymbol{\theta})=\sigma^{2}(\boldsymbol{\theta
})\int_{[-\pi,\pi]^{2}}\Psi_w(\boldsymbol{\lambda},\boldsymbol{\theta})
\frac{\partial^{2}}{\partial \theta_{i}\partial \theta_{j}}\log
\Psi_w(\boldsymbol{\lambda},\boldsymbol{\theta})d\boldsymbol{\lambda},\]
where  $\Psi_w(\boldsymbol{\lambda},\boldsymbol{\theta})=f(\boldsymbol{\lambda },\boldsymbol{\theta })w(\boldsymbol{\lambda })/\sigma^{2}(\boldsymbol{\theta
}).$ 

By (\ref{spectral_density}), (\ref{weightfunction}),  and Lemma~\ref{lem2} the function  $\Psi_w(\boldsymbol{\lambda},\boldsymbol{\theta})$ is integrable, i.e. there is  a constant $C$ such that $\Psi_w(\boldsymbol{\lambda},\boldsymbol{\theta})/C$ is a density. Hence, 
$S(\boldsymbol{\theta })=C\sigma^{2}(\boldsymbol{\theta
}){\cal I}(\theta),$
where $\cal I$ is the Fisher information matrix of the random vector $\tilde{\boldsymbol{X}}$ with the density $\tilde{\Psi}_w(\boldsymbol{\lambda},\boldsymbol{\theta})=\Psi_w(\boldsymbol{\lambda},\boldsymbol{\theta})/\int_{[-\pi,\pi]^{2}}\Psi_w(\boldsymbol{\lambda},\boldsymbol{\theta})
d\boldsymbol{\lambda}.$ Therefore, $S(\boldsymbol{\theta })$ is non-negative definite.
Note that  ${\cal I}(\theta)=-\left(E\left(\tilde{Q}_i\tilde{Q}_j\right)\right)_{i,j=1,2},$ where $\tilde{Q}_i=\frac{\partial}{\partial \theta_{i}}
\tilde{\Psi}_w(\tilde{\boldsymbol{X}},\boldsymbol{\theta}).$ The random variables $\tilde{Q}_1$ and $\tilde{Q}_2$ are not a.s. linearly related which implies positive definiteness of $S(\boldsymbol{\theta })$ . 

The entries of  $A(\boldsymbol{\theta })$ can be rewritten as
\[a_{i,j}(\boldsymbol{\theta})=8\pi^{2} \sigma^{4}(\boldsymbol{\theta})
\displaystyle\int_{[-\pi,\pi]^{2}} w^{2}(\boldsymbol{\lambda})
\frac{\partial}{\partial \theta_{i}}
\Psi(\boldsymbol{\lambda},\boldsymbol{\theta})
 \frac{\partial}{\partial \theta_{j}} \Psi(\boldsymbol{\lambda},\boldsymbol{\theta})
 \, d\boldsymbol{\lambda} =
C\sigma^{4}(\boldsymbol{\theta})\,E\left(Q_iQ_j\right)
,\]
where $Q_i=\frac{\partial}{\partial \theta_{i}}
\Psi(\boldsymbol{X},\boldsymbol{\theta})$ and the random vector $\boldsymbol{X}$ has the density $\frac{w^2(\boldsymbol{\lambda })}{\int_{[-\pi,\pi]^{2}} w^{2}(\boldsymbol{\lambda})
 \, d\boldsymbol{\lambda}}.$ 

As  $\left(E\left(Q_iQ_j\right)\right)_{i,j=1,2}$ is a non-negative definite matrix, $A(\boldsymbol{\theta })$ is non-negative definite too. Moreover, 
it is positive definite, because  the random variables $Q_1$ and $Q_2$ are not a.s. linearly related. 

Now we compute elements of the matrix $S(\boldsymbol{\theta }).$ By (\ref{sij}) 
\begin{eqnarray}
s_{i,j}(\boldsymbol{\theta}) &=& \sigma^{2}(\boldsymbol{\theta
})\int_{[-\pi,\pi]^{2}}w(\boldsymbol{\lambda}) \left[
\frac{\partial^{2}}{\partial \theta_{i}\partial \theta_{j}}
\Psi(\boldsymbol{\lambda},\boldsymbol{\theta}) -
\frac{1}{\Psi(\boldsymbol{\lambda},\boldsymbol{\theta})}
  \frac{\partial}{\partial \theta_{i}} \Psi(\boldsymbol{\lambda},\boldsymbol{\theta})
  \frac{\partial}{\partial \theta_{j}} \Psi(\boldsymbol{\lambda},\boldsymbol{\theta})
\right]d\boldsymbol{\lambda}\nonumber\\
&=& \sigma^{2}(\boldsymbol{\theta
})\int_{[-\pi,\pi]^{2}}\left(w(\boldsymbol{\lambda}) \frac{\partial}
{\partial \theta_{j}} \left[ \frac{
 \left[\frac{\partial} {\partial \theta_{i}} f(\boldsymbol{\lambda},\boldsymbol{\theta})\right]  \sigma^{2}(\boldsymbol{\theta }) -
 \left[\frac{\partial} {\partial \theta_{i}} \sigma^{2}(\boldsymbol{\theta }) \right] f(\boldsymbol{\lambda},\boldsymbol{\theta})}
 { \sigma^{4}(\boldsymbol{\theta })}
 \right]\right.\nonumber\\
& &- \frac{w(\boldsymbol{\lambda})}{f(\boldsymbol{\lambda},\boldsymbol{\theta})\sigma^{6}(\boldsymbol{\theta }) }
 \left( \left[\frac{\partial}{\partial \theta_{i}} f(\boldsymbol{\lambda},\boldsymbol{\theta}) \right] \sigma^{2}(\boldsymbol{\theta }) -
\left[\frac{\partial}{\partial \theta_{i}} \sigma^{2}(\boldsymbol{\theta }) \right] f(\boldsymbol{\lambda},\boldsymbol{\theta}) \right)\nonumber\\
& &\left. \times \left( \left[\frac{\partial}{\partial
\theta_{j}} f(\boldsymbol{\lambda},\boldsymbol{\theta}) \right]
\sigma^{2}(\boldsymbol{\theta }) - \left[\frac{\partial}{\partial
\theta_{j}} \sigma^{2}(\boldsymbol{\theta })\right]
f(\boldsymbol{\lambda},\boldsymbol{\theta})
\right)\right)d\boldsymbol{\lambda}
 \nonumber
 \end{eqnarray}
  \begin{eqnarray}
&=&
\int_{[-\pi,\pi]^{2}}\left(\frac{w(\boldsymbol{\lambda})}{\sigma^{2}(\boldsymbol{\theta
})} \left(
  \left[\frac{\partial^{2}} {\partial \theta_{j}\partial \theta_{i}}
  f(\boldsymbol{\lambda},\boldsymbol{\theta})\right]  \sigma^{2}(\boldsymbol{\theta
}) +
 \left[\frac{\partial} {\partial \theta_{j}} \sigma^{2}(\boldsymbol{\theta
})\right]
 \left[\frac{\partial} {\partial \theta}_{i} f(\boldsymbol{\lambda},\boldsymbol{\theta})\right] \right)\right. \nonumber\\
& &
 -\frac{w(\boldsymbol{\lambda})}{\sigma^{2}(\boldsymbol{\theta
})} \left(
 \left[\frac{\partial^{2}} {\partial \theta_{j}\partial \theta_{i}}
  \sigma^{2}(\boldsymbol{\theta
}) \right] f(\boldsymbol{\lambda},\boldsymbol{\theta})-
 \left[\frac{\partial} {\partial \theta_{i}} \sigma^{2}(\boldsymbol{\theta
}) \right]
 \left[\frac{\partial} {\partial \theta_{j}} f(\boldsymbol{\lambda},\boldsymbol{\theta})\right]  \right)\nonumber \\
& & 
-2\frac{w(\boldsymbol{\lambda})}{\sigma^{6}(\boldsymbol{\theta
})}\left( \frac{\partial} {\partial \theta_{j}}
\sigma^{2}(\boldsymbol{\theta }) \left( \left[\frac{\partial}
{\partial \theta_{i}}
f(\boldsymbol{\lambda},\boldsymbol{\theta})\right]\sigma^{2}(\boldsymbol{\theta
}) - \left[\frac{\partial} {\partial \theta_{i}}
\sigma^{2}(\boldsymbol{\theta }) \right]
f(\boldsymbol{\lambda},\boldsymbol{\theta})
\right) \right)\nonumber \\
& & -
\frac{w(\boldsymbol{\lambda})}{\sigma^{4}(\boldsymbol{\theta
})f(\boldsymbol{\lambda},\boldsymbol{\theta})}
\left(\left[\frac{\partial}{\partial \theta_{i}}
f(\boldsymbol{\lambda},\boldsymbol{\theta}) \right]
\sigma^{2}(\boldsymbol{\theta }) - \left[\frac{\partial}{\partial
\theta_{i}}\sigma^{2}(\boldsymbol{\theta }) \right]
f(\boldsymbol{\lambda},\boldsymbol{\theta})\right)
 \nonumber\\
& & \times \left(\left[\frac{\partial}{\partial \theta_{j}} f(\boldsymbol{\lambda},\boldsymbol{\theta}) \right]
\sigma^{2}(\boldsymbol{\theta }) - \left[\frac{\partial}{\partial
\theta_{j}} \sigma^{2}(\boldsymbol{\theta }) \right]
f(\boldsymbol{\lambda},\boldsymbol{\theta})\right)d\boldsymbol{\lambda}. \nonumber
\end{eqnarray}
By (\ref{eq_sigma_factorization}), (\ref{dtheta}), (\ref{d2theta}), and (\ref{deriv_f_d}) we obtain
\[s_{i,j}(\boldsymbol{\theta})=3\int_{[-\pi,\pi]^{2}}\frac{w(\boldsymbol{\lambda})}
{\sigma^{2}(\boldsymbol{\theta })} \left[\frac{\partial} {\partial
\theta_{j}} \sigma^{2}(\boldsymbol{\theta }) \right]
 \left[\frac{\partial} {\partial \theta}_{i} f(\boldsymbol{\lambda},\boldsymbol{\theta})\right]d\boldsymbol{\lambda}\]
 \[-\int_{[-\pi,\pi]^{2}}\frac{w(\boldsymbol{\lambda})}{f(\boldsymbol{\lambda},\boldsymbol{\theta})}
 \left[\frac{\partial}{\partial \theta_{i}}
 f(\boldsymbol{\lambda},\boldsymbol{\theta}) \right]
 \left[\frac{\partial}{\partial \theta_{j}} f(\boldsymbol{\lambda},\boldsymbol{\theta}) \right]d\boldsymbol{\lambda}= \frac{3}{\sigma^{2}(\boldsymbol{\theta })} \left[\frac{\partial} {\partial
  \theta_{j}} \sigma^{2}(\boldsymbol{\theta }) \right] \left[\frac{\partial} {\partial
  \theta_{i}} \sigma^{2}(\boldsymbol{\theta }) \right]\]
   \[-4\int_{[-\pi,\pi]^{2}} \log\left|2\cos\lambda_{i}-2u_{i}\right|\log\left|2\cos\lambda_{j}-2u_{j}\right|\, w(\boldsymbol{\lambda}) f(\boldsymbol{\lambda},\boldsymbol{\theta})d\boldsymbol{\lambda}. \]

By (\ref{aij}) the elements of the matrix $A(\boldsymbol{\theta })$ are
\begin{eqnarray}
a_{i,j}(\boldsymbol{\theta}) &=& 8\pi^2\sigma^{4}(\boldsymbol{\theta
})\int_{[-\pi,\pi]^{2}}w^2(\boldsymbol{\lambda}) 
  \frac{\partial}{\partial \theta_{i}} \Psi(\boldsymbol{\lambda},\boldsymbol{\theta})
  \frac{\partial}{\partial \theta_{j}} \Psi(\boldsymbol{\lambda},\boldsymbol{\theta})
\,d\boldsymbol{\lambda}\nonumber\\
&=& 8\pi^2 \int_{[-\pi,\pi]^{2}}w^2(\boldsymbol{\lambda}) 
\left(\left[\frac{\partial}{\partial \theta_{i}}
f(\boldsymbol{\lambda},\boldsymbol{\theta}) \right]
\sigma^{2}(\boldsymbol{\theta }) - \left[\frac{\partial}{\partial
\theta_{i}}\sigma^{2}(\boldsymbol{\theta }) \right]
f(\boldsymbol{\lambda},\boldsymbol{\theta})\right)
 \nonumber\\
& & \times \left(\left[\frac{\partial}{\partial \theta_{j}} f(\boldsymbol{\lambda},\boldsymbol{\theta}) \right]
\sigma^{2}(\boldsymbol{\theta }) - \left[\frac{\partial}{\partial
\theta_{j}} \sigma^{2}(\boldsymbol{\theta }) \right]
f(\boldsymbol{\lambda},\boldsymbol{\theta})\right)\,d\boldsymbol{\lambda}. \nonumber
\end{eqnarray}
Hence, by (\ref{deriv_f_d}) we get
$a_{i,j}(\boldsymbol{\theta}) =S_1-S_2(i,j)-S_2(j,i)+S_3,$
where
\[
S_1= 32\pi^2\sigma^{4}(\boldsymbol{\theta
})\int_{[-\pi,\pi]^{2}}\log\left|2\cos\lambda_{i}-2u_{i}\right|\log\left|2\cos\lambda_{j}-2u_{j}\right|\, w^2(\boldsymbol{\lambda}) f^2(\boldsymbol{\lambda},\boldsymbol{\theta})d\boldsymbol{\lambda},\vspace{-4mm}\]
\begin{eqnarray} S_2(i,j) &=& 16\pi^2\sigma^{2}(\boldsymbol{\theta
})\left[\frac{\partial}{\partial
\theta_{j}} \sigma^{2}(\boldsymbol{\theta }) \right]\int_{[-\pi,\pi]^{2}}\log\left|2\cos\lambda_{i}-2u_{i}\right|\, w^2(\boldsymbol{\lambda}) f^2(\boldsymbol{\lambda},\boldsymbol{\theta})d\boldsymbol{\lambda}, \nonumber\\
S_3 &=& 8\pi^2\left[\frac{\partial}{\partial
\theta_{j}} \sigma^{2}(\boldsymbol{\theta }) \right]\left[\frac{\partial}{\partial
\theta_{j}} \sigma^{2}(\boldsymbol{\theta }) \right]\int_{[-\pi,\pi]^{2}} w^2(\boldsymbol{\lambda}) f^2(\boldsymbol{\lambda},\boldsymbol{\theta})d\boldsymbol{\lambda}.\nonumber\end{eqnarray}

The proof of the condition {\bf A9} is based on the approach in \citet{Bentkus:1972}. Notice, that by (\ref{weightfunction}) there exists  a factorization $w(\boldsymbol{\lambda})=w_1(\boldsymbol{\lambda})\cdot w_2(\boldsymbol{\lambda})$ of $w(\boldsymbol{\lambda})$ such that both $\tilde{f}(\boldsymbol{\lambda},\boldsymbol{\theta}_0)=f(\boldsymbol{\lambda},\boldsymbol{\theta}_0)w_1(\boldsymbol{\lambda})$  and $w_2(\boldsymbol{\lambda})$ are bounded functions of~$\boldsymbol{\lambda}.$ For example, one can select the function $w_1(\boldsymbol{\lambda})$ to be equal a product of $|\lambda_1-\nu_1|^{2d_1}|\lambda_2-\nu_2|^{2d_2}$ and a positive smooth function on $[-\pi,\pi]^{2}.$ Let us denote $\tilde{w}_2(\boldsymbol{\lambda},\boldsymbol{\theta})=w_2(\boldsymbol{\lambda})\frac{\partial}{\partial \theta_{i}} \log
\Psi(\boldsymbol{\lambda},\boldsymbol{\theta}).$ Then, 
\[T\displaystyle\int_{[-\pi,\pi]^2}
(EI^{*}_{T}(\boldsymbol{\lambda})-f(\boldsymbol{\lambda},\boldsymbol{\theta}_0)
) w(\boldsymbol{\lambda}) \frac{\partial}{\partial \theta_{i}} \log
\Psi(\boldsymbol{\lambda},\boldsymbol{\theta}) \,
d\boldsymbol{\lambda} =T\displaystyle\int_{[-\pi,\pi]^2}
(EI^{*}_{T}(\boldsymbol{\lambda})w_1(\boldsymbol{\lambda})\]
\[-\tilde{f}(\boldsymbol{\lambda},\boldsymbol{\theta}_0)
)\tilde{w}_2(\boldsymbol{\lambda},\boldsymbol{\theta})\,
d\boldsymbol{\lambda} =T\displaystyle\int_{[-\pi,\pi]^2}
(EI^{*}_{T}(\boldsymbol{\lambda})w_1(\boldsymbol{\lambda})-E\tilde{I}^{*}_{T}(\boldsymbol{\lambda})
)\tilde{w}_2(\boldsymbol{\lambda},\boldsymbol{\theta})\,
d\boldsymbol{\lambda}\]
\begin{equation}\label{integ}+T\displaystyle\int_{[-\pi,\pi]^2}
(E\tilde{I}^{*}_{T}(\boldsymbol{\lambda})-\tilde{f}(\boldsymbol{\lambda},\boldsymbol{\theta}_0)
)\tilde{w}_2(\boldsymbol{\lambda},\boldsymbol{\theta})\,
d\boldsymbol{\lambda},\end{equation}
where $\tilde{I}^{*}_{T}(\boldsymbol{\lambda})$ denotes the unbiased  periodogram of the  random field with the spectral density~$\tilde{f}(\boldsymbol{\lambda},\boldsymbol{\theta}_0).$ 

Notice that  $\tilde{f}(\boldsymbol{\lambda},\boldsymbol{\theta}_0)$ is bounded. Hence, by the first statement of Proposition~2 in \citet{Guyon:1982} the last integral in (\ref{integ}) is ${\cal O}(T^{-2})$ and the second term in (\ref{integ}) vanishes when $T\to \infty.$  Therefore, to prove the condition {\bf A9}, it is enough to show that the first term in (\ref{integ}) vanishes too.

Let $\tilde{\gamma}(\boldsymbol{t},\boldsymbol{\theta}_0)$ denote the auto-covariance function of the random field with the spectral density~$\tilde{f}(\boldsymbol{\lambda},\boldsymbol{\theta}_0).$ By multidimensional Parseval's theorem, see  \citet{Brychkov:1992}, we get
\[\displaystyle\int_{[-\pi,\pi]^2}
(EI^{*}_{T}(\boldsymbol{\lambda})w_1(\boldsymbol{\lambda})-E\tilde{I}^{*}_{T}(\boldsymbol{\lambda})
)\tilde{w}_2(\boldsymbol{\lambda},\boldsymbol{\theta})\,
d\boldsymbol{\lambda}\]
\[= \frac{1}{(2\pi)^{2}}\displaystyle\int_{[-\pi,\pi]^2}
\Bigl(w_1(\boldsymbol{\lambda})
\sum_{t_{1}=1-T}^{T-1}\sum_{t_{2}=1-T}^{T-1} e^{-
i(\lambda_{1}t_{1}+\lambda_{2}t_{2})} {\gamma}(\boldsymbol{t},\boldsymbol{\theta}_0)\]
\[-\sum_{t_{1}=1-T}^{T-1}\sum_{t_{2}=1-T}^{T-1} e^{-
i(\lambda_{1}t_{1}+\lambda_{2}t_{2})} \tilde{\gamma}(\boldsymbol{t},\boldsymbol{\theta}_0)\Bigr) \tilde{w}_2(\boldsymbol{\lambda},\boldsymbol{\theta})\,
d\boldsymbol{\lambda}
\]
\[=
\frac{1}{(2\pi)^{2}}\displaystyle\int_{[-\pi,\pi]^2}
\Bigl(w_1(\boldsymbol{\lambda})
\sum_{(t_{1},t_{2})\in \mathbb{Z}^2} e^{-
i(\lambda_{1}t_{1}+\lambda_{2}t_{2})} {\gamma}(\boldsymbol{t},\boldsymbol{\theta}_0)I_{[1-T,T-1]^2}(t_{1},t_{2})
\]\[-\sum_{(t_{1},t_{2})\in \mathbb{Z}^2} e^{-
i(\lambda_{1}t_{1}+\lambda_{2}t_{2})} \tilde{\gamma}(\boldsymbol{t},\boldsymbol{\theta}_0)I_{[1-T,T-1]^2}(t_{1},t_{2})\Bigr)\tilde{w}_2(\boldsymbol{\lambda},\boldsymbol{\theta})\,
d\boldsymbol{\lambda}
\]
\[=\displaystyle\int_{[-\pi,\pi]^2}
\Bigl(w_1(\boldsymbol{\lambda})\displaystyle\int_{[-\pi,\pi]^2}
{f}(\boldsymbol{x},\boldsymbol{\theta}_0)\Phi_{T-1}(\boldsymbol{\lambda}-\boldsymbol{x})\,d\boldsymbol{x}\]
\[
-\int_{[-\pi,\pi]^2}
\tilde{f}(\boldsymbol{x},\boldsymbol{\theta}_0)\Phi_{T-1}(\boldsymbol{\lambda}-\boldsymbol{x})\,d\boldsymbol{x}\Bigr)\tilde{w}_2(\boldsymbol{\lambda},\boldsymbol{\theta})\,
d\boldsymbol{\lambda}
\]
\begin{equation}\label{secint} =\displaystyle\int_{[-\pi,\pi]^2}
\Bigl(\displaystyle\int_{[-\pi,\pi]^2}
{f}(\boldsymbol{x}+\boldsymbol{\lambda},\boldsymbol{\theta}_0)\Phi_{T-1}(\boldsymbol{x})(w_1(\boldsymbol{\lambda})-w_1(\boldsymbol{\lambda}+\boldsymbol{x}))\,d\boldsymbol{x}\Bigr)\tilde{w}_2(\boldsymbol{\lambda},\boldsymbol{\theta})\,
d\boldsymbol{\lambda},
\end{equation}
where $\Phi_{T-1}(\boldsymbol{x})=\frac{1}{(2\pi(T-1))^2}\left(\frac{\sin ((T-1)x_1/2)}{\sin (x_1/2)}\right)^2\left(\frac{\sin ((T-1)x_2/2)}{\sin (x_2/2)}\right)^2$ is the Fej\'{e}r kernel, $I_{[1-T,T-1]^2}(t_{1},t_{2})$ is the indicator function of the cube $[1-T,T-1]^2.$

Let $f(\boldsymbol{\lambda},\boldsymbol{\theta}_0)\in L^2([-\pi,\pi]^2),$ i.e. $(d_1,d_2)\in (0,1/4)^2.$ Then, one can refine the approach in Theorems~2.1  and~2.2 by  \citet{Bentkus:1972} for the two-dimensional case with $p=2.$ Namely, let us split the inner integral in (\ref{secint}) into two parts: the first integral is over the region $A_\alpha:=\{(x_1,x_2): |x_1|\le T^{-\alpha}, |x_2|\le T^{-\alpha}\}$ and the second integral is over $[\pi,\pi]\setminus A_{\alpha}.$ 

For $\alpha \in (2\max(d_1,d_2),1/2)$, by the choice of $w_1(\boldsymbol{\lambda})$ and
\[\int_{[\pi,\pi]\setminus A_\alpha}\Phi_{T-1}(\boldsymbol{x})\,d\boldsymbol{x}=\left(\frac{1}{{\pi(T-1)}}\int_{T^{-\alpha}}^\pi\left(\frac{\sin ((T-1)x_1/2)}{\sin (x_1/2)}\right)^2 dx_1\right)^2\]
\[\le \frac{1}{\pi^2(T-1)^2}\left(\int_{T^{-\alpha}}^\pi\frac{dx_1}{\left(\sin (x_1/2)\right)^2}\right)^2 \sim \frac{C}{T^{2\left(1-\alpha\right)}}, \quad T\to \infty ,\]
 the both above integrals are bounded by $C\varepsilon_T/T,$ where $\varepsilon_T\to 0$ when $T\to \infty.$   It implies that first term in (\ref{integ}) vanishes and completes the proof of the condition {\bf A9}.\hfill $\blacksquare$
\end{proof}

\section{SIMULATION STUDIES}

In this section we present some numerical results to confirm the theoretical findings.

Figure~\ref{fig3a} demonstrates a series of box plots to characterize the sample distribution of MCEs of the parameters $d_i,$ $i=1,2,$ as a function of $T.$  To compute it Monte Carlo simulations of the Gegenbauer field  with 100 replications for each $T=10, 20, 30, 40, 50$ were performed.
For  the parameters $u_1=0.4,$ $u_2= 0.3,$ $d_1= 0.2,$ $d_2= 0.3,$  and $\sigma^{2}_{\varepsilon}=1$ realizations of $Y_{t_1,t_2}$ were simulated using the truncated sum $\sum_{n_{1}=0}^{40} \sum_{n_{2}=0}^{40}$ in (\ref{gegenb}). For example,  a realization of the Gegenbauer random field on a $100\times100$ grid is shown in Figure~\ref{fig1}. We set the parameter values of the weight function in (\ref{weightfunction}) to $a_1=a_2=2$ and $w_0(\boldsymbol{\lambda})\equiv 1.$ The periodogram $I_{T}(\boldsymbol{\lambda})$ was computed and the minimizing argument $\hat{\boldsymbol{\theta}}_{T}$ of the functional  $\hat{U}_{T}(\boldsymbol{\theta})$ was found numerically for each simulation. 
 Figure~\ref{fig3a} demonstrates that the sample distribution of $\hat{\boldsymbol{\theta}}_{T}$ converges to ${\boldsymbol{\theta}}_{0}$  as $T$ increases. The plot of the sample probabilities $P_0(|\hat{\boldsymbol{\theta}}_{T}-{\boldsymbol{\theta}}_{0}|<\varepsilon)$ in Figure~\ref{fig3b} also confirms convergence in probability of $\hat{\boldsymbol{\theta}}_{T}$ to ${\boldsymbol{\theta}}_{0}.$ 
 \begin{figure}
  \begin{center}
\begin{minipage}{5.4cm}
                 \includegraphics[width=\textwidth]{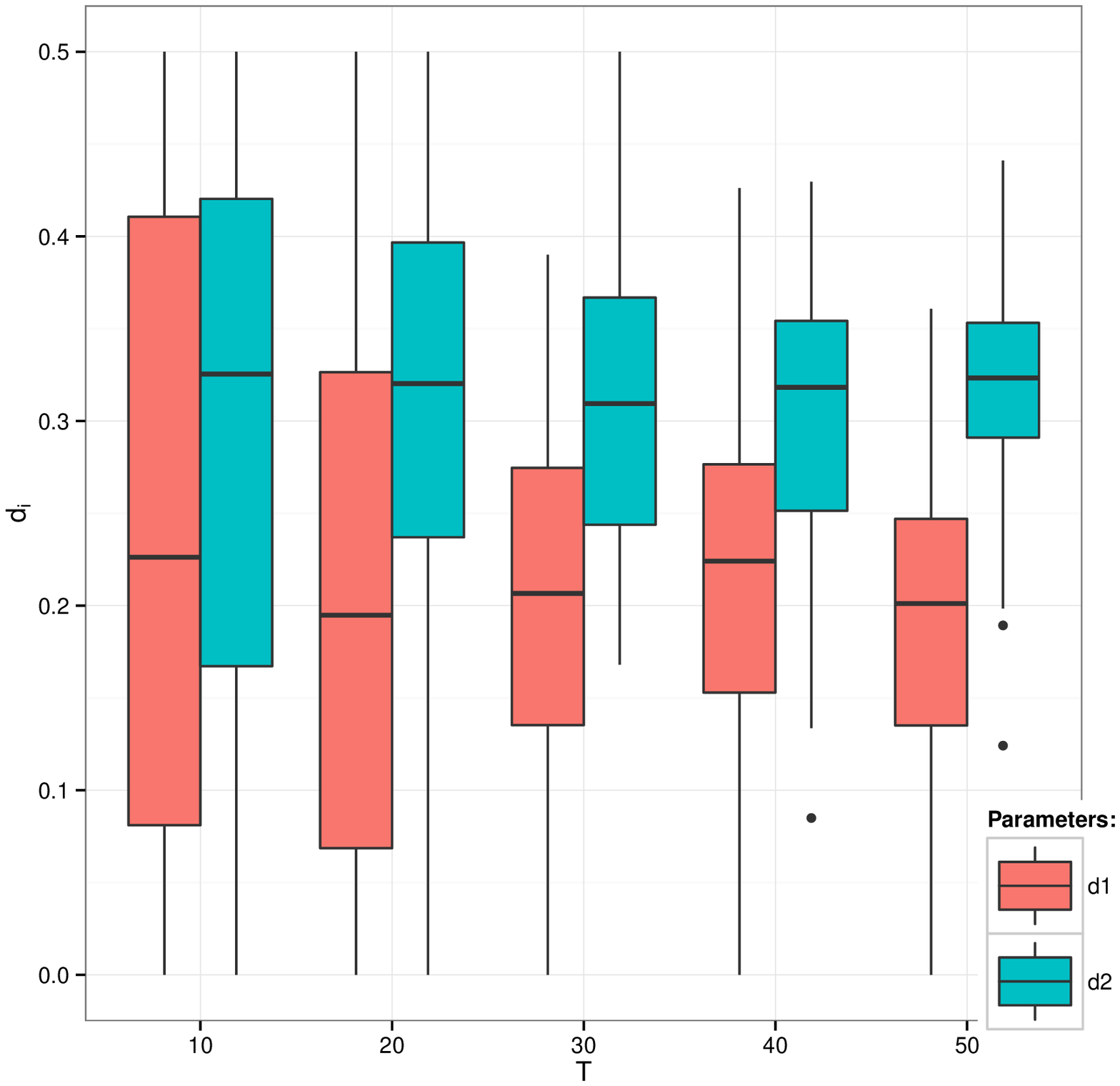}
                 \caption{Boxplots of sampled values of $\hat{\boldsymbol{\theta}}_{T}.$}
                 \label{fig3a}
         \end{minipage}%
         \quad \quad
         \begin{minipage}{5.4cm}
                 \includegraphics[width=\textwidth]{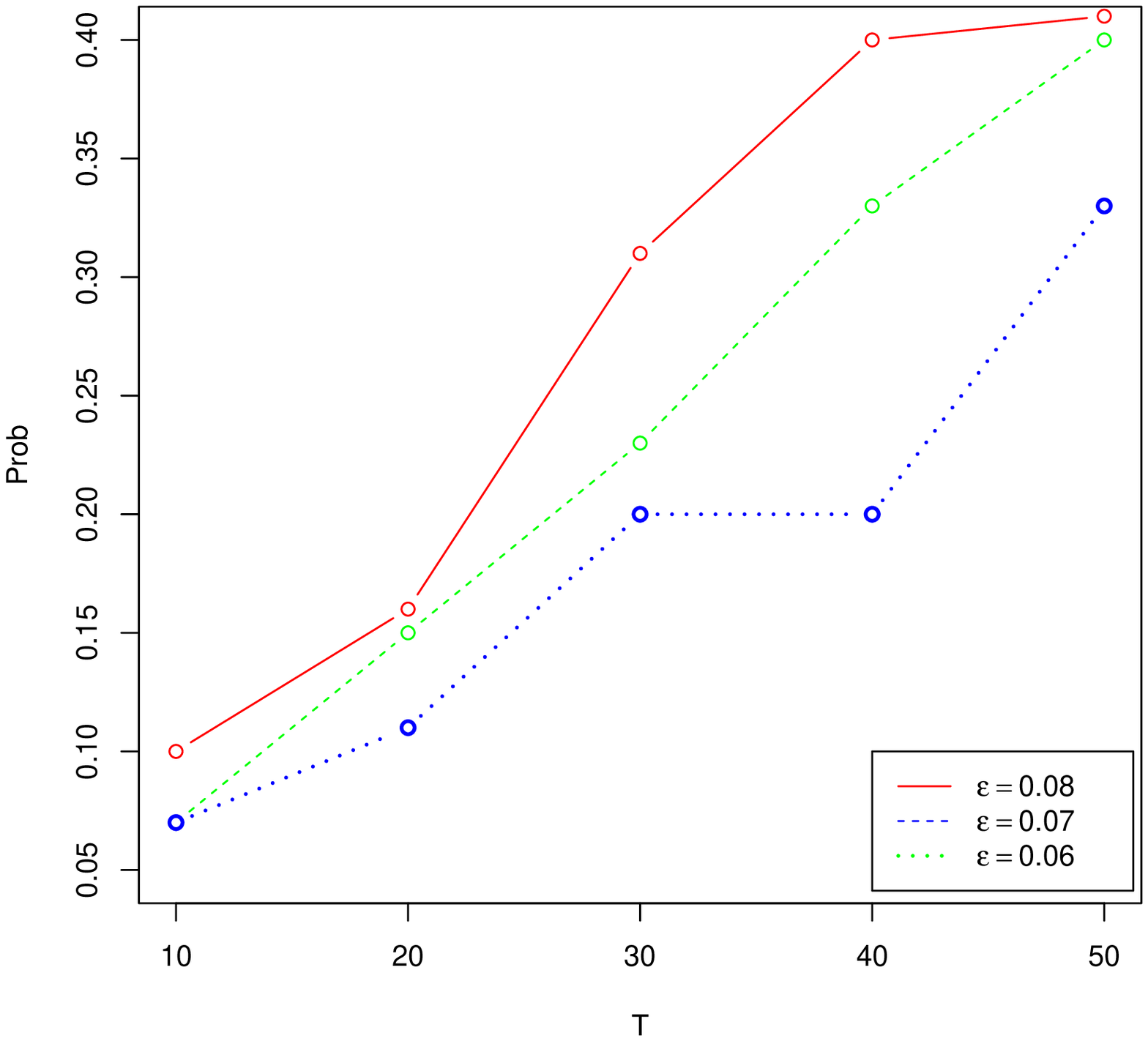}
                 \caption{Sample probabilities $P_0(|\hat{\boldsymbol{\theta}}_{T}-{\boldsymbol{\theta}}_{0}|<\varepsilon).$}
                 \label{fig3b}
         \end{minipage}
  \end{center}\vspace{-5mm}
\end{figure}

For each generated realization we also computed the value of $\hat{\sigma}^{2}_{T}$ using $I_{T}(\boldsymbol{\lambda}).$ Analogously to Figures~\ref{fig3a},~\ref{fig3b}, plots in Figures~\ref{fig4a},~\ref{fig4b} support convergence in probability of $\hat{\sigma}^{2}_{T}$ to $\sigma^{2}(\boldsymbol{\theta_{0}})$ when $T$ increases. Notice that by (\ref{eq_sigma_factorization})   we get $\sigma^{2}(\boldsymbol{\theta_{0}})\approx 74.736$ for the selected parameters. The larger values of $\varepsilon$ in Figure~\ref{fig4b} comparing to Figure~\ref{fig3b} are due to the difference in the scales for the parameters (small values measured in decimals) and variances (large values measured in tens).

 \begin{figure}
  \begin{center}
\begin{minipage}{5.4cm}
                 \includegraphics[width=\textwidth]{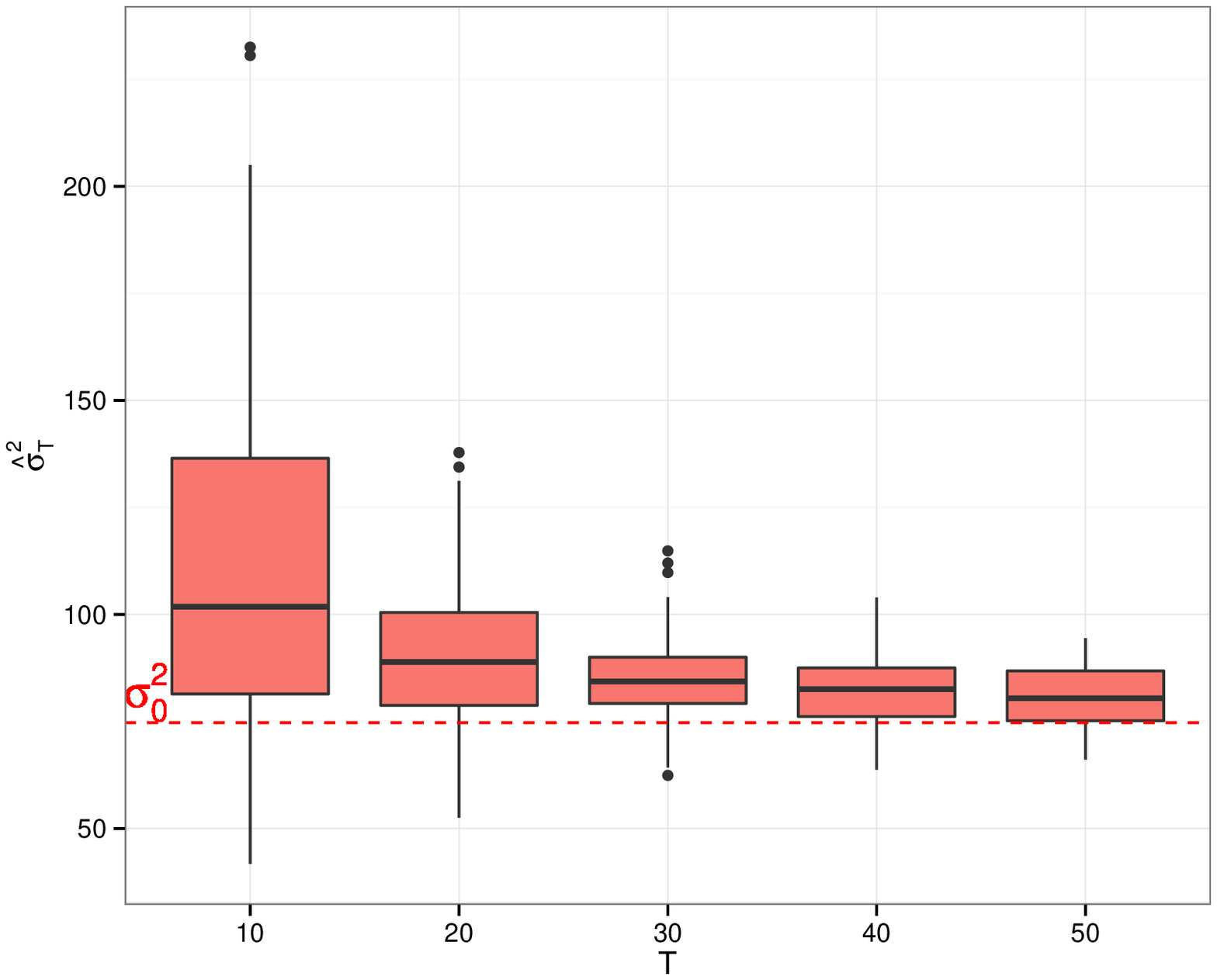}
                 \caption{Boxplots of sampled values of $\hat{\sigma}^{2}_{T}.$}
                 \label{fig4a}
         \end{minipage}%
         \quad
\begin{minipage}{5.4cm}
                 \includegraphics[width=\textwidth, trim=0 20 0 0, clip=true]{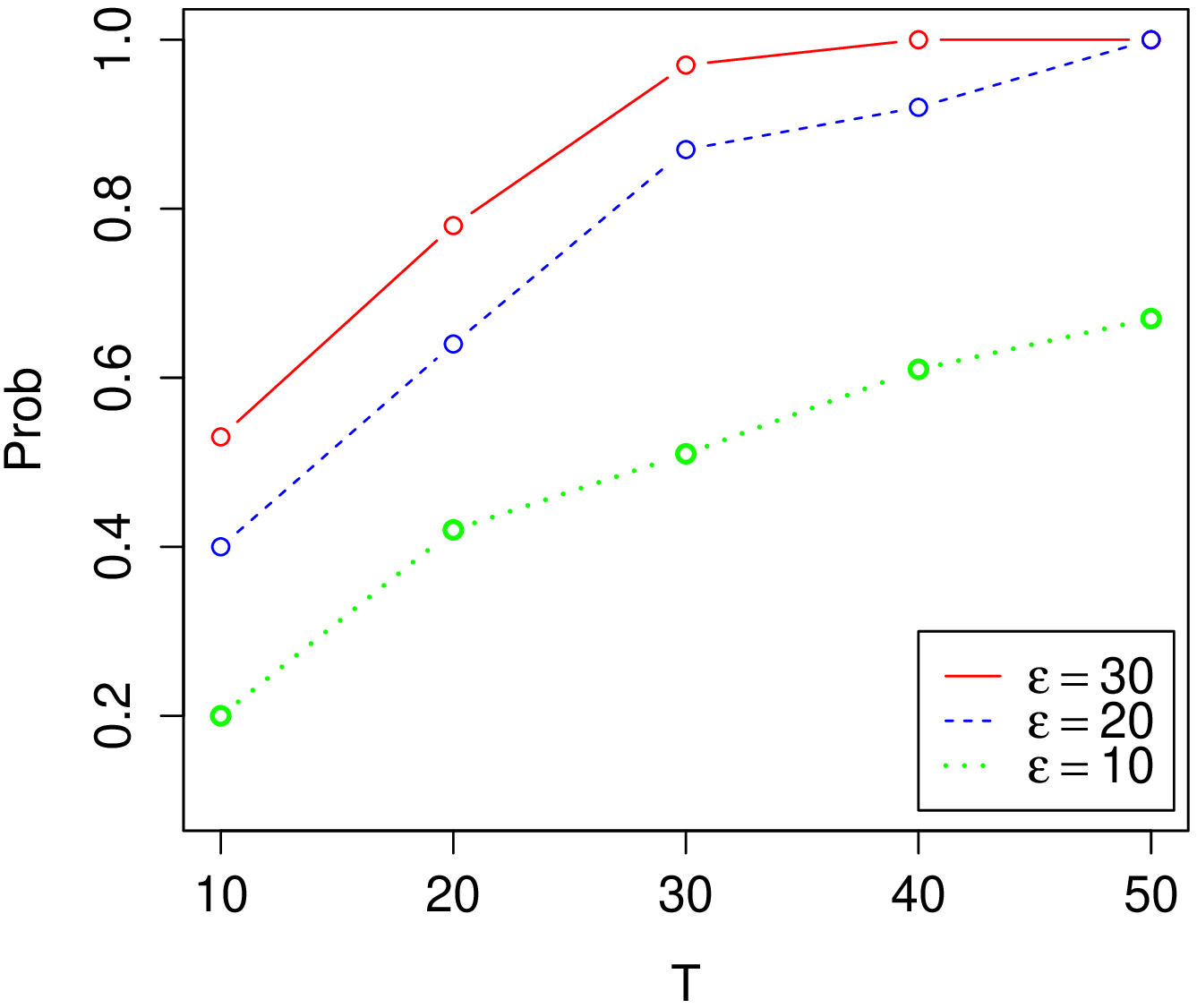}
                 \caption{Sample probabilities $P_0(|\hat{\sigma}^{2}_{T}-\sigma^{2}(\boldsymbol{\theta_{0}})|<\varepsilon).$}
                 \label{fig4b}
         \end{minipage}
%        \caption{Sample distributions of $\hat{\sigma}^{2}_{T}$.}\label{fig4}
  \end{center}
\end{figure}
 
To verify the result of Theorem~\ref{Theorem_2} we used sample values   $\hat{\boldsymbol{\theta}}^*_{50}$ which minimized the functional $\hat{U}_{T}^{*}(\boldsymbol{\theta})$ for each simulation. To avoid possible negative values the modified periodogram $I^{*}_{T}(\boldsymbol{\lambda})$ was truncated at zero by the R program. Bearing in mind the edge effect and  modified periodogram's correction, Figures~\ref{fig5a},~\ref{fig5b}  demonstrate that  the results are close to the expected ones even for the relatively small $T=50.$  The normal Q-Q plot of each component of $\hat{\boldsymbol{\theta}}_{50}^*$ in Figures~\ref{fig5b}   matches  with the theoretical normal distribution.
To test the bivariate normality hypothesis about $\hat{\boldsymbol{\theta}}_{50}^*$ we used the Shapiro-Wilk, energy, and  kurtosis tests of multivariate normality from the R packages {\sc mvnormtest, energy,} and {ICS}. In all the tests, p-values ($0.9491,$ $0.4605,$ and $0.5314$) confirmed that $\hat{\boldsymbol{\theta}}_{T}^*$ asymptotically follows a bivariate normal distribution. Simulations for other values of the parameters were run, with similar results.

 \begin{figure}
  \begin{center}
\begin{minipage}{5.4cm}
                 \includegraphics[width=\textwidth]{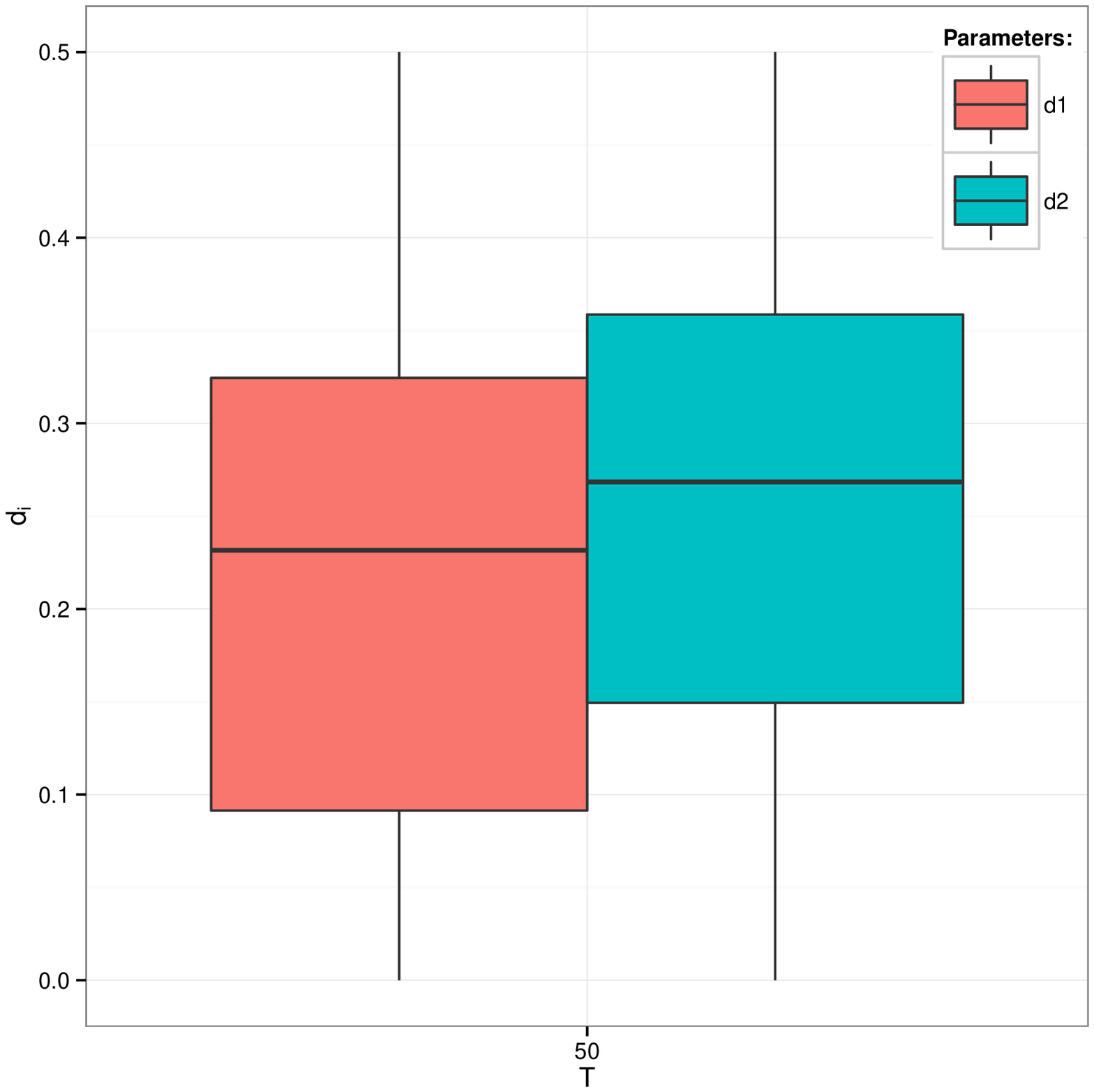}
                 \caption{Boxplots of sampled values of $\hat{\boldsymbol{\theta}}_{50}^*.$}
                 \label{fig5a}
         \end{minipage}%
         \quad \quad
\begin{minipage}{5.4cm}
                 \includegraphics[width=\textwidth]{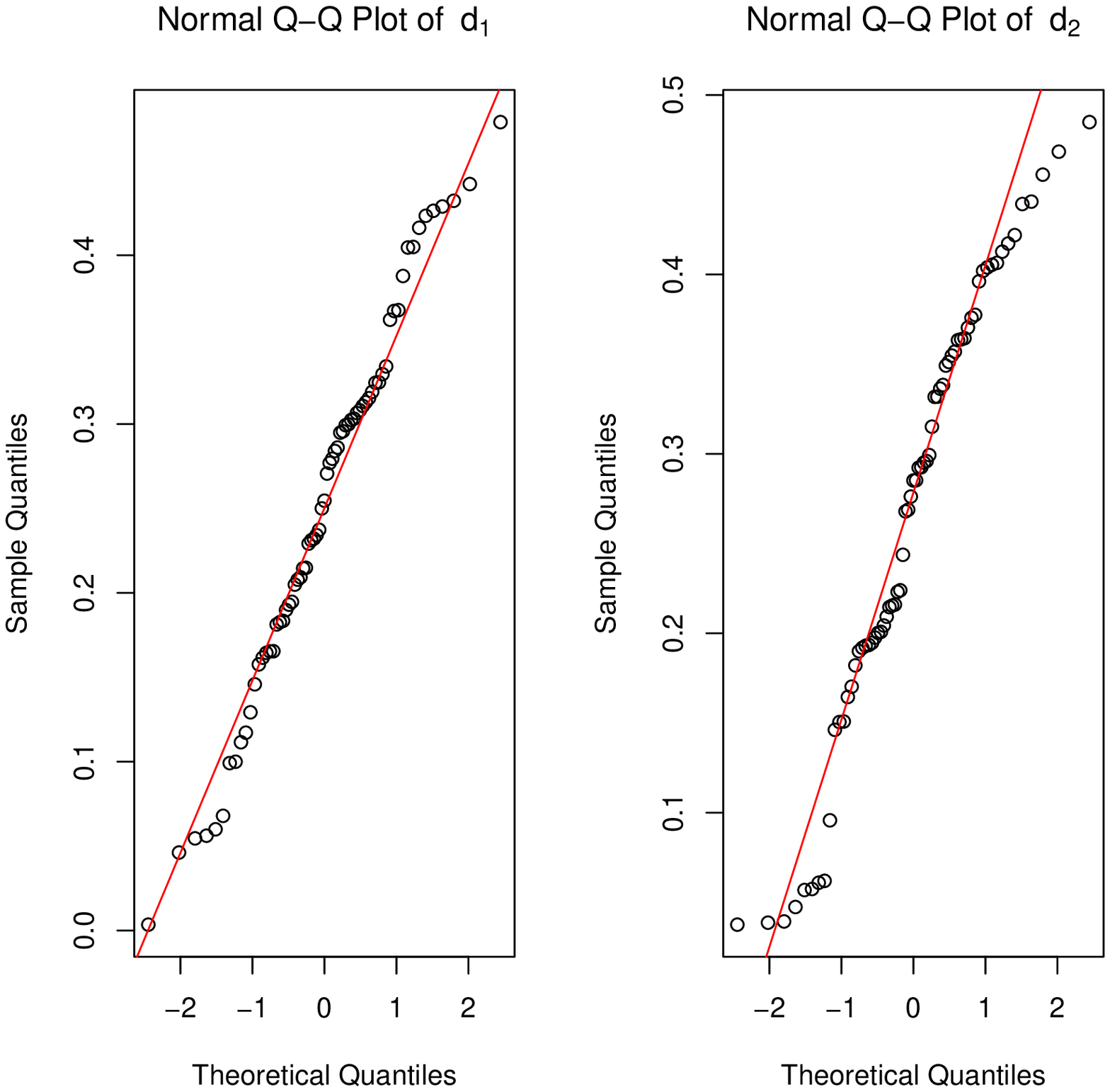}
                 \caption{Normal Q-Q plots for each component of $\hat{\boldsymbol{\theta}}^*_{50}.$}
                 \label{fig5b}
         \end{minipage}
%        \caption{Sample distributions of adjusted MCEs.} \label{fig5}
  \end{center}
\end{figure}

Hence, we conclude that the MCEs are consistent estimators and the distributions of $\hat{\boldsymbol{\theta}}_{T}^*$
converge to the bivariate normal law. Note that the simulation studies not only comply with the obtained results  for $(d_1,d_2)\in (0,1/4)^2,$  but also indicate that the theoretical results may be extended to all possible values of $(d_1,d_2)$ in  $(0,1/2)^2.$

\section{DIRECTIONS FOR FUTURE RESEARCH}
The estimation methodology based on the unbiased periodogram was
introduced in \citet{Guyon:1982,Guyon:1995}, see also \citet{Heyde:1993}.
Recently, the paper by \citet{Robinson:2006} and the references therein provided a detailed discussion on the topic. It studied mainly the difficulties arising in the application of the methodology in high dimensions. In particular, they investigated  problems
arising in relation to non-uniformly increasing domain asymptotics
associated with different expansion rates of the studied domain in each spatial direction. In this paper, we considered the case $d=2$ and restricted out attention to the case
of  uniformly increasing domain asymptotics.  The case of non-uniformly increasing domain asymptotics is left for future investigations. 

An extended version of the derived results can be obtained for more
general formulations of the unbiased periodogram. In particular,
different growing rates can be allowed for each spatial dimension in the definition of the sampling area.  For example, one can consider the following  generalized
version  of the two-dimensional unbiased periodogram (see, for
example, \citet{Robinson:2006})\vspace{-3mm}
\[
I_{g}(\lambda_{1},\lambda_{2})= \frac{1}{(2\pi)^{2}}
\sum_{t_{1}=1-g_{1}(T)}^{g_{1}(T)-1}\sum_{t_{2}=1-g_{2}(T)}^{g_{2}(T)-1}
e^{- i(\lambda_{1}t_{1}+\lambda_{2}t_{2})}
\hat{\gamma}_{T}(\boldsymbol{t}), \nonumber
\]
\noindent \nopagebreak where the functions $g_{i}(T),$ $ i=1,2,$ satisfy some suitable conditions (for example, $g_{i}(T)\rightarrow \infty,$ $T\rightarrow \infty,$  
and $g_{i}(T)\leq C T,$ $C<1,$ for $i=1,2,$ and sufficiently large  $T$).

An important area for future explorations is to extend the results of \citet{Anh:2004} and simultaneously estimate the locations of singularities and long-range dependence parameters using the MCE methodology.  A feasible way to approach this problem would be relaxing the $L_2$-integrability assumptions in conditions {\bf A5-A8.}

It also would be interesting to extend the methodology by \citet{Bentkus:1972} to prove the condition {\bf A9} for all $(d_1,d_2)$ in  $(0,1/2)^2.$

Note that our simulation results show that the proposed minimum
contrast estimation methodology works in the case of 
uniformly increasing domain asymptotics. 

\

\noindent\textbf{Supplementary Materials}
The codes used for simulations in this paper are available from the site 
\verb! https://googledrive.com/host/!

\noindent
\verb!0B7UxM8o_bnBxdG9zNU9MdHFOQUU/MCE%20Gegenbauer%20fields/!

\begin{acknowledgements}
The authors are grateful for the referee's careful reading of the paper and many detailed comments and suggestions, which helped to improve the paper. Also, this work has been supported in part by projects  MTM2012-32674 of the DGI, MEC, and P09-FQM-5052 of the Andalousian CICE, Spain.  A.Olenko was partially supported by the 2013 LTU research grant.
\end{acknowledgements}
 
\section*{Appendix}
The conditions for consistency and asymptotic
normality of the MCE for parameters of stationary
fractional Riesz-Bessel type random fields given in
\citet{Anh:2004} are specified below for random fields on  $\mathbb{Z}^{2}.$

\begin{description}
  \item[\bf A1.] Let $Y_{t_{1},t_{2}},\ \boldsymbol{t}=(t_{1}, t_{2}) \in  \mathbb{Z}^{2},$
   be  a real-valued measurable stationary Gaussian random field  with zero mean and a spectral density
   $f(\boldsymbol{\lambda},\boldsymbol{\theta}),$  where
   $\boldsymbol{\lambda}=(\lambda_{1},\lambda_{2})\in [-\pi,\pi]^{2},$
   $\boldsymbol{\theta} \in \Theta,$ and $\Theta$ is a compact set.
    Assume that  $\boldsymbol{\theta}_{0} \in \mbox{int}(\Theta ),$ where $\boldsymbol{\theta}_{0}$ is the true value of the
    parameter vector $\boldsymbol{\theta}.$

  \item[\bf A2.] If $\boldsymbol{\theta_{1}} \neq \boldsymbol{\theta_{2}}$ then  $f(\boldsymbol{\lambda}, \boldsymbol{\theta_{1}})
  \neq f(\boldsymbol{\lambda}, \boldsymbol{\theta_{2}})$ for almost all $\boldsymbol{\lambda} \in [-\pi,\pi]^{2}$
  with respect to the Lebesgue measure.

  \item[\bf A3.] There exists a nonnegative function $w(\boldsymbol{\lambda})$,  $\boldsymbol{\lambda} \in [-\pi,\pi]^{2}$, such that
      \begin{enumerate}
        \item $w(\boldsymbol{\lambda})$ is symmetric about $(0,0),$ i.e. $w(\boldsymbol{\lambda})=w(-\boldsymbol{\lambda});$
        \item $w(\boldsymbol{\lambda}) f(\boldsymbol{\lambda},\boldsymbol{\theta}) \in L_{1}\left([-\pi,\pi]^{2}\right)$
        for all $\boldsymbol{\theta} \in \Theta.$ 
      \end{enumerate}
  \item[\bf A4.] The derivatives $\nabla_{\boldsymbol{\theta}}\Psi(\boldsymbol{\lambda},\boldsymbol{\theta})$
  exist and  it is legitimate to differentiate
  under the integral sign in equation (\ref{equation_psi}), i.e.
        \begin{equation}
            \nabla_{\boldsymbol{\theta}} \displaystyle\int_{[-\pi,\pi]^{2}}
     \Psi(\boldsymbol{\lambda},\boldsymbol{\theta}) w(\boldsymbol{\lambda})\,\, d\boldsymbol{\lambda} =
            \displaystyle\int_{[-\pi,\pi]^{2}} \nabla_{\boldsymbol{\theta}}
     \Psi(\boldsymbol{\lambda},\boldsymbol{\theta}) w(\boldsymbol{\lambda})\,\, d\boldsymbol{\lambda} = 0. \nonumber
        \end{equation}

  \item[\bf A5.] For all $\boldsymbol{\theta }\in \Theta$ the function $w(\boldsymbol{\lambda})$, $
  \boldsymbol{\lambda}\in[-\pi,\pi]^{2},$
   satisfies
      $$f(\boldsymbol{\lambda},\boldsymbol{\theta_{0}})w(\boldsymbol{\lambda})
      \log \Psi(\boldsymbol{\lambda},\boldsymbol{\theta}) \in L_{1}\left( [-\pi,\pi]^{2}\right) \cap L_{2}\left( [-\pi,\pi]^{2}\right).$$

  \item[\bf A6.] There exists a function $\upsilon(\boldsymbol{\lambda}),$ $\boldsymbol{\lambda}\in [-\pi,\pi]^{2},$
   such that
        \begin{enumerate}
            \item the function $h(\boldsymbol{\lambda},\boldsymbol{\theta})=\upsilon(\boldsymbol{\lambda})
            \log\Psi(\boldsymbol{\lambda},\boldsymbol{\theta})$ is uniformly continuous on $[-\pi,\pi]^{2} \times \Theta;$
            \item $f(\boldsymbol{\lambda},\boldsymbol{\theta}_{0}) w(\boldsymbol{\lambda}) / \upsilon(\boldsymbol{\lambda}) \in
             L_{1}([-\pi,\pi]^{2}) \cap L_{2}([-\pi,\pi]^{2}).$
\end{enumerate}
\item[\bf A7.] The function $\Psi(\boldsymbol{\lambda},\boldsymbol{\theta})$ is twice differentiable on  $\Theta$ and
        \begin{enumerate}
            \item $f(\boldsymbol{\lambda},\boldsymbol{\theta}_{0}) w(\boldsymbol{\lambda}) \frac{\partial^{2}}
            {\partial \theta_{i}\partial \theta_{j}} \log \Psi(\boldsymbol{\lambda},\boldsymbol{\theta})
             \in L_{1}([-\pi,\pi]^{2}) \bigcap L_{2}([-\pi,\pi]^{2}),$ for all $i,j,$ and $\boldsymbol{\theta} \in \Theta;$
            \item $f(\boldsymbol{\lambda},\boldsymbol{\theta}_{0}) w(\boldsymbol{\lambda}) \frac{\partial}
            {\partial \theta_{i}} \log \Psi(\boldsymbol{\lambda},\boldsymbol{\theta}) \in L_{k}([-\pi,\pi]^{2}),$
            for all $i,$ $\boldsymbol{\theta} \in \Theta,$   and $k\geq 1.$
        \end{enumerate}

    \item[\bf A8.] The matrices $\mathbf{S}(\boldsymbol{\theta})=(s_{ij}(\boldsymbol{\theta}))$      and $\mathbf{A}(\boldsymbol{\theta})=(a_{ij}(\boldsymbol{\theta}))$ with the elements defined by {\rm(\ref{sij})} and {\rm(\ref{aij})} are positive definite. 
    
    \item[\bf A9.] The spectral density $f(\boldsymbol{\lambda},\boldsymbol{\theta}),$
    the weight function $w(\boldsymbol{\lambda}),$ and the function
    $ \frac{\partial}{\partial \theta_{i}} \log \Psi(\boldsymbol{\lambda},\boldsymbol{\theta})$  are such that for all $i$ and $\boldsymbol{\theta}\in \Theta:$
$$T\displaystyle\int_{[-\pi,\pi]^2}
(EI^{*}_{T}(\boldsymbol{\lambda})-f(\boldsymbol{\lambda},\boldsymbol{\theta}_0)
) w(\boldsymbol{\lambda}) \frac{\partial}{\partial \theta_{i}} \log
\Psi(\boldsymbol{\lambda},\boldsymbol{\theta}) \,\,
d\boldsymbol{\lambda} \longrightarrow 0, \quad\text{as} \quad
T\longrightarrow \infty.$$
\end{description}


\begin{thebibliography}{}

\bibitem[Abramowitz and Stegun(1972)]{Abra:1972} Abramowitz M, Stegun  IA, eds. (1972) {Handbook of mathematical functions with formulas, graphs, and mathematical tables.} Dover, New York 

\bibitem[Anh et al.(1999)]{Anh:1999}
    Anh VV, Angulo JM, Ruiz-Medina MD (1999)
    Possible long-range dependence in fractional random fields.
    {J Stat Plan Infer} {80}:95--110

\bibitem[Anh et al.(2004)]{Anh:2004}
   Anh  VV, Leonenko NN, Sakhno LM (2004)
    On a class of minimum contrast estimators for fractional stochastic processes and fields.
    {J Stat Plan Infer} {123}:161--185

\bibitem[Anh et al.(2007)]{Anh:2007}
     Anh  VV, Leonenko NN, Sakhno LM (2007)
    Minimum contrast estimation of random processes based on information of second and third orders.
    {J Stat Plan Infer} {137}(4):1302--1331

\bibitem[Anh and Lunney(1995)]{Anh:1995}
   Anh  VV, Lunney  KE (1995)
    Parameter estimation of random fields with long-range depedence.
    {Math Comput Model} {21}:67--77

\bibitem[Arteche and Robinson(2000)]{Arteche:2000} Arteche  J, Robinson PM (2000) Semiparametric inference in seasonal and cyclical long memory processes. {J Time Ser Anal} {21}(1):1--25

\bibitem[Basu and Reinsel(1993)]{Basu:1993}
    Basu S, Reinsel  GC (1993)
    Properties of the spatial unilateral first-order $ARMA$ model.
    {Adv Appl Probab} {25}(3):631--648

\bibitem[Bentkus(1972)]{Bentkus:1972} Bentkus R (1972)  The error in estimating the spectral function of a stationary process. {Litovsk Mat Sb} {12}(1):55--71
 
\bibitem[Beran et al.(2009)]{Beran:2009}
    Beran J, Ghosh S, Schell  D (2009)
    On least squares estimation for long-memory lattice processes.
    {J Multivariate Anal} {100}:2178--2194

\bibitem[Boissy et al.(2005)]{Boissy:2005}
   Boissy  Y, Bhattacharyya BB, Li X, Richardson GD (2005)
    Parameter estimates for fractional autoregressive spatial process.
    {Ann Statist} {33}:2553--2567

\bibitem[Brychkov et al.(1992)]{Brychkov:1992} Brychkov YuA, Glaeske  HJ, Prudnikov AP,  V\~{u} KT  (1992)
Multidimensional integral transformations. Gordon and Breach Science Publishers, New York


\bibitem[Chan and Tsai(2012)]{Chan:2012}
    Chan KS, Tsai H (2012)
    Inference of seasonal long-memory aggregate time series.
    {Bernoulli} {4}(18):1448--1464


\bibitem[Chung(1996a)]{Chung:1996a}
    Chung CF~(1996)
    A generalized fractionally integrated autoregressive moving-average process.
    {J Time Ser Anal} {17}:111--140

\bibitem[Chung(1996b)]{Chung:1996b}
     Chung CF~(1996)
    Estimating a generalized long memory process.
    {J Econometrics} {73}:237--259

\bibitem[Cohen and Francos(2002)]{Cohen:2002}
   Cohen  G, Francos  JM (2002)
    Linear least squares estimation of regression models for two-dimensional random fields.
    {J Multivariate Anal} {82}:431--444 

\bibitem[Collet and Fadili(2006)]{Collet:2006}
    Collet JJ and Fadili MJ (2006)
    {Simulation of Gegenbauer Processes using Wavelet Packets.}
    arXiv:math/0608613

\bibitem[Espejo et al.(2014)]{Espejo:2014}
    Espejo RM, Leonenko N, Ruiz-Medina MD (2014)
    Gegenbauer random fields.
    {Random Oper Stoch Equ}  {22}(1): 1--16

\bibitem[Ferrara and Gu\'{e}gan(2001)]{Ferrara:2001}
   Ferrara L, Gu\'{e}gan D (2001)
    Comparison of parameter estimation methods in cyclical long memory time series. In:
   Junis C, Moody J, Timmermann A (eds.) {Development in Forecast Combination and Portfolio Choice,} Chapter~8,  Wiley, New York

\bibitem[Giraitis et al.(2001)]{Giraitis:2001}
    Giraitis L, Hidalgo J, Robinson PM (2001)
    Gaussian estimation of parametric spectral density with uknown pole.
    {Ann Statist} {29}:987--1023

\bibitem[Gradshteyn and Ryzhik(1980)]{Gradshteyn:1980}
    Gradshteyn IS, Ryzhik IM (1980)
    {Table of integrals, series, and products.}
    Elsevier, Burlington


\bibitem[Gray et al.(1989)]{Gray:1989}
    Gray HL, Zhang NF, Woodward WA (1989)
    On generalized fractional processes.
    {J Time Ser Anal} {10} , 233--257

\bibitem[Guo et al.(2009)]{Guo:2009}
    Guo H, Lim CY, Meerschaert MM (2009)
    Local Whittle estimator for anisotropic random field,
    {J Multivariate Anal} {100}:993--1028

\bibitem[Guyon(1982)]{Guyon:1982}
    Guyon X (1982)  Parameter estimation for a stationary process on a d-dimensional lattice. {Biometrika}, {69} , 95--105.

\bibitem[Guyon(1995)]{Guyon:1995}
    Guyon X (1995)
    {Random fields on a network.}
    Springer-Verlag, New York


\bibitem[Heyde and Gay(1993)]{Heyde:1993} Heyde  CC,  Gay R (1993)  Smoothed periodogram asymptotics  and estimation for processes and fields with possible
long-range dependence. {Stochastic Process Appl} {45}:169--182

\bibitem[Hsu and Tsai(2009)]{Hsu:2009}
    Hsu NJ, Tsai H (2009)
    Semiparametric estimation for seasonal long-memory time series using generalized exponential models.
    {J Stat Plan Infer} {139}:1992--2009

\bibitem[Ivanov et al.(2013)]{Iva:2013}   Ivanov  AV, Leonenko  NN, Ruiz-Medina  MD, Savich IN (2013) Limit theorems for weighted non-linear transformations of Gaussian processes with singular spectra. {Ann Probab} {41}(2):1088--1114 

\bibitem[Leonenko and Olenko(2013)]{LeoOle:2013} Leonenko N, Olenko A (2013) Tauberian and Abelian theorems for long-range dependent random fields,  {Methodol Comput Appl Probab} 15(4): 715--742  

\bibitem[Leonenko and Sakhno(2006)]{Leonenko:2006}
    Leonenko NN, Sakhno LM (2006)
    On the Whittle estimators for some classes of continuous-parameter random processes and fields.
    {Stat Probabil Lett} {76}:781--795

\bibitem[Li and McLeod(1986)]{Li:1986}
   Li  WK, McLeod AI (1986)
    Fractional time series modelling.
    {Biometrika} {73}:217--221

\bibitem[McElroy and Holan(2012)]{McElroy:2012}
    McElroy TS, and Holan SH (2012)
    On the computation of autocovariances for generalized Gegenbauer processes. {Statist  Sinica} {22}:1661--1687

\bibitem[Olenko(2013)]{Ole:2013} Olenko  A (2013)  Limit theorems for weighted functionals of cyclical long-range dependent random fields.  {Stoch Anal Appl} {31}(2):199--213

\bibitem[Reisen et al.(2006)]{Reisen:2006}
    Reisen V, Rodrigues AL, Palma W (2006)
    Estimation of seasonal fractionally integrated processes.
    {Comput Statist Data Anal} {50}:568--582
    
\bibitem[Robinson and Sanz(2006)]{Robinson:2006} Robinson  PM,  Sanz JV (2006)  Modified Whittle   estimation of multilateral models on a lattice. {J Multivariate Anal} {97}:1090-1120

\bibitem[Schilling(2005)]{Schilling:2005}  Schilling RL (2005)
Measures, integrals and martingales. Cambridge University Press, New York

\bibitem[Taniguchi(1987)]{Taniguchi:1987}
    Taniguchi M  (1987)
    Minimum contrast estimation for spectral densities of stationary processes.
    {J R Stat Soc Ser B Methodol} {49}(3):315--325

\bibitem[Vidal-Sanz(2009)]{Sanz:2009} Vidal-Sanz JM (2009) Automatic spectral density estimation for random fields on a lattice via bootstrap. {Test} {18}(1):96--114

\bibitem[WeiLin et al.(2012)]{WeiLin:2012}
   WeiLin  X, WeiGuo Z, XiLi Z (2012)
    Minimum contrast estimator for fractional Ornstein-Uhlenbeck processes.
    {Sci China Math} {55}(7):1497--1511

\bibitem[Woodward et al.(1998)]{Woodward:1998}
    Woodward WA, Cheng QC, Gray HL (1998)
    A $k$-factor $GARMA$ long-memory model.
    {J Time Ser Anal} {19}:485--504

\bibitem[Yao and Brockwell(2006)]{Yao:2006} Yao Q,  Brockwell PJ (2006) Gaussian maximum likelihood estimation for ARMA models II: Spatial processes. {Bernoulli}: {12}(3):403--429 
\end{thebibliography}
\end{document}